\documentclass[a4paper]{article}

\usepackage[english]{babel}
\usepackage[utf8x]{inputenc}
\usepackage[T1]{fontenc}
\newcommand{\nm}{\noalign{\smallskip}}
\newcommand{\ds}{\displaystyle}
\usepackage{graphicx}
\usepackage{subcaption}
\usepackage{tikz}
\usepackage{pgfplots}
\usepackage{wrapfig}
\usepackage{cutwin}
\graphicspath{ {images/} }
\usetikzlibrary{intersections,decorations.pathreplacing,decorations.markings,calc}

\usepackage[a4paper,top=3cm,bottom=2cm,left=2.5cm,right=2.5cm,marginparwidth=1.75cm]{geometry}

\usepackage{amsmath}
\usepackage{amsthm}
\usepackage{amssymb}
\newtheorem{prop}{Propostition}
\newtheorem{thm}{Theorem}
\newtheorem{lem}{Lemma}

\theoremstyle{definition}
\newtheorem{rmk}{Remark}
\usepackage{tikz}
\usepackage{graphicx}
\usepackage{hyperref}
\usepackage{caption}
\usepackage{float}
\theoremstyle{definition}

\newcommand{\Z}{\mathbb{Z}}
\newcommand{\R}{\mathbb{R}}
\newcommand{\C}{\mathcal{C}}
\newcommand{\A}{\mathcal{A}}

\newcommand{\F}{\mathcal{F}}
\newcommand{\p}{\partial}
\renewcommand{\L}{\mathcal{L}}
\renewcommand{\S}{\mathcal{S}}
\newcommand{\K}{\mathcal{K}}
\renewcommand{\O}{\mathcal{O}}
\renewcommand{\epsilon}{\varepsilon}
\renewcommand{\phi}{\varphi}
\newcommand{\dx}{\: \mathrm{d}}

\newcommand{\eqnref}[1]{(\ref {#1})}
\def\nm{\noalign{\medskip}}
\def\capacity{{\mathrm{Cap}}}

\title{Subwavelength localized modes for acoustic waves in bubbly crystals with a defect}

\author{ Habib Ammari\thanks{\footnotesize Department of Mathematics, 
ETH Z\"urich, 
R\"amistrasse 101, CH-8092 Z\"urich, Switzerland (habib.ammari@math.ethz.ch, brian.fitzpatrick@sam.math.ethz.ch, sanghyeon.yu@sam.math.ethz.ch).} \and Brian Fitzpatrick\footnotemark[1] \and Erik Orvehed Hiltunen\thanks{\footnotesize Engineering physics, Uppsala University, Sweden (Erik.Orvehed$\_$Hiltunen.0207@student.uu.se).} \and Sanghyeon Yu\footnotemark[1]}
\date{}

\begin{document}
	\maketitle

\begin{abstract}
The ability to control wave propagation is of fundamental interest in many areas of physics. Photonic and phononic crystals have proved very useful for this purpose but, because they are based on Bragg interference, these artificial media require structures with large dimensions. 
In [Ammari et al., Subwavelength phononic bandgap opening in bubbly media, J. Diff. Eq., 263 (2017), 5610--5629], it has been 
proved that a subwavelength bandgap opening occurs in bubbly phononic crystals. To demonstrate the opening of a subwavelength phononic bandgap,  a periodic arrangement of bubbles is considered and their subwavelength Minnaert resonance is exploited.
In this paper, this subwavelength bandgap is used to demonstrate cavities, very similar to those obtained in photonic and phononic crystals, albeit of deeply subwavelength dimensions. 
The key idea is to perturb the size of a single bubble inside the crystal, thus creating a defect. The goal is then to analytically and numerically show that this crystal has a localized eigenmode close to the defect bubble.
\end{abstract}

\def\keywords2{\vspace{.5em}{\textbf{  Mathematics Subject Classification
(MSC2000).}~\,\relax}}
\def\endkeywords2{\par}
\keywords2{35R30, 35C20.}

\def\keywords{\vspace{.5em}{\textbf{ Keywords.}~\,\relax}}
\def\endkeywords{\par}
\keywords{bubble, subwavelength resonance, subwavelength phononic crystal,  subwavelength cavity.}

\section{Introduction}	
	
It is well-known in solid state physics that the periodicity of atoms composing crystals is responsible for the existence of both conducting bands and bandgaps for electrons. This property is a consequence of the Floquet-Bloch theory applied to the wave-function of
electrons. Similarly, electromagnetic and  elastic waves propagating in periodic media are subject to the same formalism giving rise to the existence of ranges of frequencies for which no propagation is allowed, so called bandgaps. Such materials are known as photonic and phononic crystals. At an interface with free space, these types of crystals act as  mirrors for incoming waves, termed Bragg mirrors. Similarly, a point defect can be created by locally modifying the properties of a crystal. This results in a cavity if the mode supported by the defect is resonant within the bandgap. Physically, the underlying mechanism is Bragg interference. In these periodic media, the Bragg condition typically occurs when the period of the medium scales with the wavelength \cite{arma, lecturenotes,  soussi, figotinkuchment, kuchment2, hempel, Lipton, Lipton2}. As a consequence, photonic and phononic crystals are typically structured with a period corresponding to half the operating wavelength. This constrains the range of applications, specifically in low-frequency regimes where the wavelength is large \cite{pnas,nature}.

Based on the phenomenon of subwavelength resonance, a class of phononic crystals that exhibit bandgaps with deep subwavelength spatial scales   have been fabricated \cite{phononic1}. In \cite{bandgap},  the opening of a subwavelength phononic bandgap in bubbly crystals has been proved. This subwavelength bandgap is mainly due to the cell resonance of the bubbles in the quasi-static regime and is quite different from the usual bandgaps in photonic/phononic crystals where the gap opens at wavelengths which, as mentioned previously, are comparable to the period of the structure. In \cite{highfrequency}, it has been further proved that the first Bloch eigenvalue achieves its maximum at the corner of the Brillouin zone. Moreover, by computing the asymptotic behaviour of the Bloch eigenfunctions in the periodic structure near that critical frequency, it has been  demonstrated that these eigenfunctions can be decomposed into two parts: one which is slowly varying and satisfies a homogenized equation, while the other is periodic across each elementary crystal cell and is varying more rapidly. The asymptotic analysis of wave fields near the critical frequency where a subwavelength bandgap opens rather than the zero frequency has been performed. This rigorously justifies  the observed superfocusing of acoustic waves in bubbly crystals near and below the maximum of the first Bloch eigenvalue and confirms the bandgap opening near and above this critical frequency. We refer the reader to \cite{superfocusing, Ammari2015, Ammari2015_a} for the mathematical analysis of the superfocusing phenomenon in resonant media.

Bubbly media is a natural model for the control of wave propagation at the deep subwavelength scale because of the simplicity of the constituent resonant structure, the air bubbles. It is well-known that a single bubble in water possesses a subwavelength resonance which is called the Minnaert resonance \cite{first, Minnaert}. 
By the subwavelength resonance, we mean that the resonator (in our case, the bubble) is of size smaller than the the wavelength.
This resonance is due to the high contrast in density between the bubble and the surrounding medium and it makes the air bubble an ideal subwavelength resonator (the bubble can be two order of magnitude smaller than the wavelength at the resonant frequency). 

In phononic crystals, a point defect can be created by locally removing a scatterer. This results in a small cavity because a resonant mode is created by the defect within the bandgap \cite{weinstein2,weinstein3,figotin1,figotin2,weinstein1,klein,lin1,lin2}. Following this concept, many components have been demonstrated based on periodic media such as waveguides using line defects \cite{santosa, santosa2}. However, because of their wavelength scale period, phononic crystals result in relatively large devices. This seriously constrains the range of applications, especially in the low frequency regime where the wavelength is large.

If we remove one bubble inside the bubbly crystal, we cannot create a defect mode. 
The defect created in this fashion is actually too small to support a resonant mode, while in a phononic crystal removing one scatterer allows for the existence of a stationary defect mode since the typical scale of such a defect is the wavelength. This illustrates a strong difference between Bragg bandgaps and subwavelength bandgaps in bubbly crystals. 

In order to tackle this issue, we have to physically introduce a resonant defect inside the crystal of subwavelength resonators, and this is achieved by simply detuning one resonator with respect to the rest of the medium. In the case of the bubbly
medium, we prove in this paper that by perturbing the radius of one bubble we create a detuned resonator with a resonance frequency
that is upward shifted, thus falling within the subwavelength bandgap. Moreover, we will show that the way to shift the frequency upwards depends on the crystal: in the dilute regime we have to decrease the defect bubble size while in the non-dilute regime we have to increase the size.

The aim of this paper is to prove the existence of this defect mode. Through the application of layer potential techniques, Floquet-Bloch theory, and Gohberg-Sigal theory we derive an original formula for the defect mode frequency, along with proving the existence of a subwavelength localized mode. Our results are complemented by several numerical examples which serve to validate them in two dimensions. Our results formally explain the experimental observations  reported  in \cite{ experiment2013, experiment} in the case of Helmholtz resonators. They lay the mathematical foundation for the analysis of wave propagation control the deeply subwavelength scale.   
Subwavelength cavities have a high quality factor and a low mode volume. These two  effects are typically associated with the enhancement of the emission rate
of an emitter or the so-called Purcell factor \cite{purcell}.

The paper is organized as follows. In Section \ref{sec-1}
we formulate the spectral problem for a bubble phononic crystal and introduce some basic results regarding the quasi-periodic Green's function, stability of the essential spectrum, and Floquet-Bloch theory. In Section \ref{sec-2} we use the 
fictitious source superposition method introduced in \cite{Wilcox}
 for modelling the defect and characterize the fictitious sources as the solution of some system of integral equations. In Section \ref{sec-3}, we prove existence of a localized defect mode and derive an asymptotic formula for the resonant frequency created inside the subwavelength bandgap by perturbing the size of a single bubble in terms of the difference between its radius and the radius of the original bubbles.  In Section \ref{sec-4} we perform numerical simulations  to illustrate the main findings of this paper.  We make use of the multipole expansion method to compute the defect mode inside the subwavelength bandgap. The paper ends with some concluding remarks.

\section{Preliminaries} \label{sec-1}
\subsection{Layer potentials}

 Let $\Gamma^0$ and $\Gamma^k,k>0$  be the fundamental solution of the Laplace and Helmholtz equations in dimension two, respectively, {\it i.e.},
\begin{equation*}
\begin{cases}
\ds \Gamma^k(x,y) = -\frac{i}{4}H_0^{(1)}(k|x-y|), \ & k>0, \\
\nm
\ds \Gamma^0(x,y) = \frac{1}{2\pi}\ln|x-y|,
\end{cases}
\end{equation*}
where $H_0^{(1)}$ is the Hankel function of the first kind and order zero.  

Let $\S_{D}^k: L^2(\partial D) \rightarrow H_{\textrm{loc}}^1(\R^2)$ be the single layer potential defined by
\begin{equation*}
\S_D^k[\phi](x) = \int_{\partial D} \Gamma^k(x,y)\phi(y) \dx \sigma(y), \quad x \in \R^2.
\end{equation*}
We also define the Neumann-Poincar\'e operator $\left(\K_D^k\right)^*: L^2(\partial D) \rightarrow L^2(\partial D)$ by
\begin{equation*}
\left(\K_D^k\right)^* = \int_{\partial D} \frac{\partial }{\partial \nu_x}\Gamma^k(x,y) \phi(y) \dx \sigma(y), \quad x \in \partial D.
\end{equation*}
In the case when $k=0$, we will omit the superscripts and write $\S_D$ and $\K_D^*$, respectively. The following so-called jump relations of $\S_D^k$ on the boundary $\partial D$ are well-known  (see, for instance, \cite{lecturenotes}):
\begin{equation*}
\S_D^k[\phi]\big|_+ = \S_D^k[\phi]\big|_-,
\end{equation*}
and
\begin{equation*}
\frac{\partial }{\partial \nu}\S_D^k[\phi] \bigg|_{\pm} =  \left(\pm\frac{1}{2} I + \left(\K_D^k\right)^*\right) [\phi].
\end{equation*}
Here, $\partial/\partial \nu$ denotes the outward normal derivative,  and $|_\pm$ denotes the limits from outside and inside $D$. In two dimensions, the fundamental solution of the free-space Helmholtz equation has a logarithmic singularity given by \cite{lecturenotes}
\begin{equation}\label{eq:hankel}
-\frac{i}{4}H_0(k|x-y|) = \frac{1}{2\pi} \ln |x-y| + \eta_k + \sum_{j=1}^\infty\left( b_j \ln(k|x-y|) + c_j \right) (k|x-y| )^{2j},
\end{equation}
where $\ln$ is the principal branch of the logarithm and
$$ \eta_k = \frac{1}{2\pi}(\ln k+\gamma-\ln 2)-\frac{i}{4}, \quad b_j=\frac{(-1)^j}{2\pi}\frac{1}{2^{2j}(j!)^2}, \quad c_j=b_j\left( \gamma - \ln 2 - \frac{i\pi}{2} - \sum_{n=1}^j \frac{1}{n} \right),$$
and $\gamma$ is the Euler constant. Define 
\begin{equation*}
\hat{\S}_D^k[\phi](x) = \S_D[\phi](x) + \eta_k\int_{\partial D} \phi\dx \sigma.
\end{equation*}
Then the following expansion holds:
\begin{equation} \label{eq:Sexpansion}
\S_D^k =  \hat{\S}_{D}^k + \O(k^2 \ln k).
\end{equation}

We also introduce a quasi-periodic version of the layer potentials.
Let $Y$ be the unit cell $[-1/2,1/2)^2$ in $\mathbb{R}^2$. For $\alpha\in [0,2\pi)^2$, the function $\Gamma^{\alpha, k}$ is defined to satisfy
$$ (\Delta_x + k^2) \Gamma^{\alpha, k} (x,y) = \sum_{n\in \mathbb{R}^2} \delta(x-y-n) e^{i n\cdot \alpha},$$
where $\delta$ is the Dirac delta function and  $\Gamma^{\alpha, k} $ is $\alpha$-quasi-periodic, {i.e.}, $e^{- i \alpha\cdot x} \Gamma^{\alpha, k}(x,y)$ is periodic in $x$ with respect to $Y$.  

We define the quasi-periodic single layer potential $\mathcal{S}_D^{\alpha,k}$ by
$$\mathcal{S}_D^{\alpha,k}[\phi](x) = \int_{\partial D} \Gamma^{\alpha,k} (x,y) \phi(y) d\sigma(y),\quad x\in \mathbb{R}^2.$$
It satisfies the following jump formulas:
\begin{equation*}
\S_D^{\alpha,k}[\phi]\big|_+ = \S_D^{\alpha,k}[\phi]\big|_-,
\end{equation*}
and
$$ \frac{\p}{\p\nu} \Big|_{\pm} \mathcal{S}_D^{\alpha,k}[\phi] = \left( \pm \frac{1}{2} I +( \mathcal{K}_D^{-\alpha,k} )^*\right)[\phi]\quad \mbox{on}~ \p D,$$
where $(\mathcal{K}_D^{-\alpha,k})^*$ is the operator given by
$$ (\mathcal{K}_D^{-\alpha, k} )^*[\phi](x)= \mbox{p.v.} \int_{\p D} \frac{\p}{\p\nu_x} \Gamma^{\alpha,k}(y,y) \phi(y) d\sigma(y).$$
We remark that it is known that $\mathcal{S}_D^{\alpha,0} : L^2(\p D) \rightarrow H^1(\p D)$ is invertible for $\alpha \ne 0$ \cite{lecturenotes}.

\subsection{Floquet transform}

A function $f(x)$ is said to be $\alpha$-quasi-periodic in the variable $x\in \R^2$
if $e^{-i\alpha\cdot x}f(x)$ is periodic. Given a function $f\in L^2(\R^2)$, the Floquet transform is defined as
\begin{equation}\label{eq:floquet}
\F[f](x,\alpha) = \sum_{m\in \Z^2} f(x-m) e^{i\alpha \cdot m},
\end{equation}
which is $\alpha$-quasi-periodic in $x$ and periodic in $\alpha$. Let $Y = [-1/2,1/2)^2$ be the unit cell and $BZ := \R^2 / 2\pi \Z^2 \simeq [0,2\pi)^2$ be the so-called first Brillouin zone. The Floquet transform is an invertible map $\F:L^2(\R^2) \rightarrow L^2(Y\times BZ)$, with inverse  (see, for instance, \cite{kuchment, lecturenotes})
\begin{equation*}
\F^{-1}[g](x) = \frac{1}{(2\pi)^2}\int_{BZ} g(x,\alpha) \dx \alpha.
\end{equation*}
%\textbf{*** Often this has a normalization factor of something like $1/(2\pi)^2$ in front. Not sure if it is missing intentionally here or not. ***}

\subsection{Bubbly crystals and subwavelength bandgaps}\label{subsec:bandgap}

Here we briefly review the subwavelength bandgap opening of a bubbly crystal from \cite{bandgap}. 

%In this subsection, we assume the unperturbed crystal $\C$.
Assume that a single bubble occupies $D$, which is a disk of radius $R<1/2$ centred at the origin. We denote by $\rho_b$ and $\kappa_b$ the density and the bulk modulus of the air inside the bubble, respectively. We let $\rho_w$ and $\kappa_w$ be the corresponding parameters for the water.  We introduce
\begin{equation*} % \label{data1}
v_w = \sqrt{\frac{\kappa_w}{\rho_w}}, \quad v_b = \sqrt{\frac{\kappa_b}{\rho_b}}, \quad k_w= \frac{\omega}{v_w} \quad \text{and} \quad k_b= \frac{\omega}{v_b}
\end{equation*}
to be the speed of sound outside and inside the bubbles, and the wavenumber outside and inside the bubbles, respectively. $\omega$ corresponds to the operating frequency of acoustic waves.
We also introduce two dimensionless contrast parameters
\begin{equation*} % \label{data2}
\delta = \frac{\rho_b}{\rho_w} \quad \text{and} \quad \tau= \frac{k_b}{k_w}= \frac{v_w}{v_b} =\sqrt{\frac{\rho_b \kappa_w}{\rho_w \kappa_b}}. 
\end{equation*}
We assume  that the wave speeds outside and inside the bubbles are comparable to each other and that there is a large contrast in the bulk moduli, that is, $$\delta \ll 1,\quad \tau= O(1).$$
In this paper, for the sake of simplicity of presentation, we shall assume $v_w=v_b=1$ .

Let $\C = \cup_{n\in\Z^2}(D+n)$ be the periodic bubbly crystal. 
Consider the following quasi-periodic scattering problem:
\begin{equation} \label{eq-scattering-quasiperiodic}
\left\{
\begin{array} {ll}
&\ds \nabla \cdot \frac{1}{\rho_w} \nabla  v + \frac{\omega^2}{\kappa_w} v  = 0 \quad \text{in} \quad Y \backslash \overline{D}, \\
\nm
&\ds \nabla \cdot \frac{1}{\rho_b} \nabla  v+ \frac{\omega^2}{\kappa_b} v  = 0 \quad \text{in} \quad D, \\
\nm
&\ds  v|_{+} -v|_{-}  =0   \quad \text{on} \quad \partial D, \\
\nm
& \ds  \frac{1}{\rho_w} \frac{\p v}{\p \nu} \bigg|_{+} - \frac{1}{\rho_b} \frac{\p v}{\p \nu} \bigg|_{-} =0 \quad \text{on} \quad \partial D,\\
\nm
&  e^{-i \alpha \cdot x} v  \,\,\,  \mbox{is periodic.}
  \end{array}
 \right.
\end{equation}

It is known that \eqref{eq-scattering-quasiperiodic} has non-trivial solutions for discrete values of $\omega$ such as
$$  0 \le \omega_1^\alpha \le \omega_2^\alpha \le \cdots $$
and thus we have the following band structure of propagating frequencies for the periodic bubbly crystal $\C$:
$$
[0,\max_\alpha \omega_1^\alpha] \cup [\min_\alpha \omega_2^\alpha, \max_\alpha \omega_2^\alpha ] \cup \cdots.
$$

In \cite{bandgap}, it is proved that there exists a subwavelength spectral gap opening in the band structure.  Let us briefly review this result. 
We look for a solution of  \eqref{eq-scattering-quasiperiodic} which has the following form:
\begin{equation} \label{Helm-solution}
v =
\begin{cases}
\mathcal{S}_{D}^{\alpha,k_w} [\varphi^\alpha]\quad & \text{in} ~ Y \setminus \bar{D},\\
 \mathcal{S}_{D}^{k_b} [\psi^\alpha]   &\text{in} ~   {D},
\end{cases}
\end{equation}
for some densities $\varphi^\alpha, \psi^\alpha \in  L^2(\p D)$. 
Using the jump relations for the single layer potentials, one can show that~\eqref{eq-scattering-quasiperiodic} is equivalent to the boundary integral equation
\begin{equation}  \label{eq-boundary}
\mathcal{A}^\alpha(\omega, \delta)[\Phi^\alpha] =0,  
\end{equation}
where
\[
\mathcal{A}^\alpha(\omega, \delta) = 
 \begin{pmatrix}
  \mathcal{S}_D^{k_b} &  -\mathcal{S}_D^{\alpha,k}  \\
  -\frac{1}{2}+ \mathcal{K}_D^{k_b, *}& -\delta( \frac{1}{2}+ (\mathcal{K}_D^{ -\alpha,k})^*)
\end{pmatrix}, 
\,\, \Phi^\alpha= 
\begin{pmatrix}
\varphi^\alpha\\
\psi^\alpha
\end{pmatrix}.
\]
%Throughout the paper, we denote by $\mathcal{H} = L^2(\p D) \times L^2(\p D)$ and by $\mathcal{H}_1 = H^1(\p D) \times L^2(\p D)$,  and use $(\cdot, \cdot)$ for the inner product in $L^2$ spaces and $\| \cdot \|$ for the norm in $\mathcal{H}$.  It is clear that $\mathcal{A}(\omega, \delta)$ is a bounded linear operator from $\mathcal{H}$ to $\mathcal{H}_1$, \textit{i.e.} $\mathcal{A}(\omega, \delta) \in \mathcal{B}(\mathcal{H}, \mathcal{H}_1)$. 
% Let us consider the limiting case when $\delta =0$. The operator $\mathcal{A}(\omega, \delta)$ is a perturbation of
%\begin{equation}  \label{eq-A_0-3d}
% \mathcal{A}^\alpha(0, 0) = 
% \begin{pmatrix}
%  \mathcal{S}_D^{k_b} &  -\mathcal{S}_D^{\alpha,k}  \\
%  -\frac{1}{2}+ \mathcal{K}_D^{k_b,*}& 0
%\end{pmatrix}.
%\end{equation} 

Since it can be shown that $\omega=0$ is a characteristic value for the operator-valued analytic function $\mathcal{A}(\omega,0)$, we can conclude the following result by the Gohberg-Sigal theory \cite{lecturenotes, Gohberg1971}.
\begin{lem}
For any $\delta$ sufficiently small, there exists a characteristic value 
$\omega_1^\alpha= \omega_1^\alpha(\delta)$ to the operator-valued analytic function 
$\mathcal{A}^\alpha(\omega, \delta)$
such that $\omega_1^\alpha(0)=0$ and 
$\omega_1^\alpha$ depends on $\delta$ continuously.
\end{lem}

The next theorem gives the asymptotic expansion of $\omega_1^\alpha$ as $\delta\rightarrow 0$.

\begin{thm}{\rm{\cite{bandgap}}} \label{approx_thm} For $\alpha \ne 0$ and sufficiently small $\delta$, we have
\begin{align}
\omega_1^\alpha= \sqrt{\frac{\delta \capacity_{D,\alpha}}{\pi R^2}} + O(\delta^{3/2}), \label{o_1_alpha}
\end{align}
where the constant $\capacity_{D,\alpha}$ is given by
$$\capacity_{D,\alpha}:= - \langle(\mathcal{S}_D^{\alpha,0})^{-1} [\chi_{\partial D}], \chi_{\partial D}\rangle.$$
\end{thm}

Let $\omega_1^*=\max_\alpha \omega_1^\alpha$.
The following theorem specifies the subwavelength bandgap opening.

\begin{thm}{\rm{\cite{bandgap}}} \label{main_bandgap}
For every $\epsilon>0$, there exists $\delta_0>0$  and $\tilde \omega > \omega_1^*$ such that 
\begin{equation}
 [ \omega_1^*+\epsilon, \tilde\omega ] \subset [\max_\alpha \omega_1^\alpha, \min_\alpha \omega_2^\alpha]
 \end{equation} 
 for $\delta<\delta_0$.
\end{thm}

%\subsection{Stability of the essential spectrum}
%We recall some facts about the essential spectrum $\sigma_\textrm{ess}(L)$ of a self-adjoint operator $L$. In this case, the discrete spectrum $\sigma_\textrm{disc}(L)$ is the set of isolated points $\lambda \in \sigma(L)$ which are eigenvalues of $L$ of finite multiplicity. The essential spectrum is then given by $\sigma_\textrm{ess}(L) = \sigma(L)\setminus \sigma_\textrm{disc}(L)$.

%Now, let $$\rho(x) = \rho_w +(\rho_b-\rho_w) \chi_\C(x)$$ and let $$\rho_d(x) = \rho_w +(\rho_b-\rho_w )\chi_{\C_d}(x).$$ Let $W = \nabla \cdot \frac{1}{\rho(x)}\cdot \nabla$ be the operator corresponding to the unperturbed crystal and let $W_d = \nabla \cdot \frac{1}{\rho_d(x)}\cdot \nabla$ be the operator corresponding to the perturbed crystal. The following result, proved in \cite{figotin}, states that the essential spectra of $W$ is preserved under the perturbation.
%\begin{prop}
%The essential spectrum of $W$ and $W_d$ coincide, i.e.,  $\sigma_\mathrm{ess}(W) = \sigma_\mathrm{ess}(W_d)$.
%\end{prop}
%This shows that if any point $\lambda \in \sigma(W_d)$ is inside the bandgap given by $W$, then $\lambda$ is an isolated eigenvalue of $W_d$ of finite multiplicity. The corresponding eigenmode will be exponentially decaying away from the defect .

\section{Integral representation for bubbly crystals with a defect} \label{sec-2}

%Here, we will formulate the boundary integral representation of the solution $u$.

\subsection{Bubbly crystals with a defect: problem formulation}

\begin{figure}[tb]
    \centering
    \begin{tikzpicture}[scale=1.5]
    
    \draw[dashed] (0,0) circle (10pt) node{$D_d$};
    \draw (0,0) circle (6pt) node[yshift=-15pt, xshift=13pt ]{$D$};
    \draw (1,0) circle (10pt);
    \draw (0,1) circle (10pt);
    \draw (1,1) circle (10pt);
    \draw (-1,0) circle (10pt);
    \draw (0,-1) circle (10pt);
    \draw (1,-1) circle (10pt);
    \draw (-1,1) circle (10pt);
    \draw (-1,-1) circle (10pt);
    \draw (1.6,0) node{$\cdots$};
    \draw (-1.55,0) node{$\cdots$};
    \draw (0,1.6) node{$\vdots$};
    \draw (0,-1.5) node{$\vdots$};
    \draw (1.7,1) node[xshift=2pt]{$\rho_w, \kappa_w$};
    \draw (1,1) node{$\rho_b, \kappa_b$};
    \draw[opacity = 0.3] (0.5,0.5) -- (0.5,-0.5) node[yshift=-2pt, xshift=5pt ]{$Y$} -- (-0.5,-0.5) -- (-0.5,0.5) -- cycle;
    \end{tikzpicture}
    \caption{Illustration of the defect crystal and the material parameters in the case of a smaller defect bubble.} \label{fig:defect}
\end{figure}
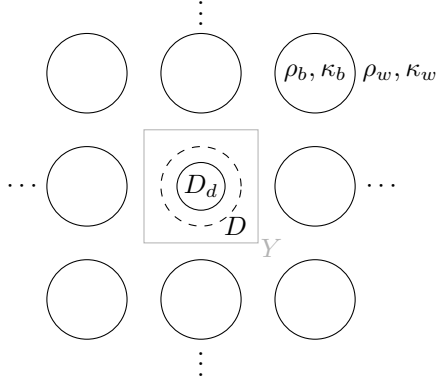

Consider now a perturbed crystal, where the central disk $D$ is replaced by a defect disk $D_d$ of radius $R_d$ with $R_d<1/2$. We will consider both the cases when $R_d > R$ and $R_d < R$. Let $\C_d = D_d \cup \left( \cup_{n\in\Z^2\setminus\{0,0\}} D+n \right)$ be the perturbed crystal and let $\epsilon = R_d-R \in (-R, 0) \cup (0,1/2-R)$ be the perturbation of the radius. We consider the following problem:
\begin{equation} \label{eq:scattering}
\left\{
\begin{array} {ll}
	&\ds \nabla \cdot \frac{1}{\rho_w} \nabla  u+ \frac{\omega^2}{\kappa_w} u  = 0 \quad \text{in} \quad \R^2 \backslash \C_d, \\
	\nm
	&\ds \nabla \cdot \frac{1}{\rho_b} \nabla  u+ \frac{\omega^2}{\kappa_b} u  = 0 \quad \text{in} \quad \C_d, \\
	\nm
	&\ds  u|_{+} -u|_{-}  =0   \quad \text{on} \quad \partial \C_d, \\
	\nm
	& \ds  \frac{1}{\rho_w} \frac{\partial u}{\partial \nu} \bigg|_{+} - \frac{1}{\rho_b} \frac{\partial u}{\partial \nu} \bigg|_{-} =0 \quad \text{on} \quad \partial \C_d .
\end{array}
\right.
\end{equation}
% Let 
%$$
%\rho(x) = \rho_w + (\rho_b - \rho_w) \chi_\C(x), \quad \kappa(x) = \kappa_w + (\kappa_b - \kappa_w) \chi_\C(x),
%$$
% where $\chi_{\C}$ is the characteristic function of $\C$.  

As discussed in Subsection \ref{subsec:bandgap}, the unperturbed problem (with $\C$ instead of $\C_d$ in \eqnref{eq:scattering}), has the following essential spectrum for the propagating frequencies:
$$
[0,\max_\alpha \omega_1^\alpha] \cup [\min_\alpha \omega_2^\alpha, \max_\alpha \omega_2^\alpha ] \cup \cdots.
$$
In fact, it can be easily shown that the perturbed crystal problem \eqnref{eq:scattering} has the same essential spectrum. This is because the essential spectrum is stable under compact perturbations \cite{figotin, MMMP4}.

In this paper, we want to show that by modifying the central bubble $D$, there exists a frequency $\omega^\epsilon$, slightly above $\max_\alpha\omega_1^\alpha$, which results in a non-trivial solution to the problem \eqnref{eq:scattering}. The solution $u$ associated with the frequency $\omega^\epsilon$ should be localized since $\omega^\epsilon$ lies in the bandgap. Moreover, it is of a subwavelength nature.

\subsection{Effective sources for the defect}

It is difficult to obtain the boundary integral formulation of the problem \eqref{eq:scattering} directly.
Here we consider the effective source solution which models the defect $D_d$ by placing non-trivial sources onto the boundary of the central bubble $D$ of the unperturbed crystal $\C$.
Since the geometry of the unperturbed crystal is periodic, we can use the Floquet-Bloch theory.
This is motivated by the fictitious source superposition method introduced in \cite{Wilcox}.

Let us consider the following problem:
\begin{equation} \label{eq:scattering_fictitious}
\left\{
\begin{array} {ll}
	&\ds \nabla \cdot \frac{1}{\rho_w} \nabla  \tilde{u}+ \frac{\omega^2}{\kappa_w} \tilde{u}  = 0 \quad \text{in} \quad \R^2 \setminus \C, \\
	\nm
	&\ds \nabla \cdot \frac{1}{\rho_b} \nabla  \tilde{u}+ \frac{\omega^2}{\kappa_b} \tilde{u}  = 0 \quad \text{in} \quad \C, \\
	\nm
	&\ds  \tilde{u}|_{+} -\tilde{u}|_{-}  =f\delta_{m,(0,0)}  \quad \text{on} \quad \partial D+m, \ m\in\mathbb{Z}^2, \\
	\nm
	& \ds  \frac{1}{\rho_w} \frac{\partial \tilde{u}}{\partial \nu} \bigg|_{+} - \frac{1}{\rho_b} \frac{\partial \tilde{u}}{\partial \nu} \bigg|_{-} =g\delta_{m,(0,0)} \quad \text{on} \quad \partial D+m, \ m\in\mathbb{Z}^2.
\end{array}
\right.
\end{equation} 
where $f,g$ are the source terms and $\delta_{m,n}$ is the Kronecker delta function.
Note that the non-zero sources are present only on the boundary of the central bubble $D$. 

Introduce the notation
$$
D_l = \begin{cases} D_d \quad & \text{if } \epsilon > 0 \\ D \quad & \text{if } \epsilon < 0, \end{cases} \qquad
D_s = \begin{cases} D \quad & \text{if } \epsilon > 0 \\ D_d \quad & \text{if } \epsilon < 0. \end{cases}
$$
Thus, $D_l$ and $D_s$ denotes the largest and smallest of $D, D_d$, respectively. If the source terms $f$ and $g$ satisfy some appropriate conditions, then we will have 
$$
u \equiv \tilde{u}  \quad \mbox{in } (\mathbb{R}^2\setminus D_l) \cup D_s.
$$
Once this is achieved, the original solution $u$ can be recovered by extending the solution $\widetilde{u}$ to the whole region including $D_l\setminus D_s$ with boundary conditions on $\p D$ and $\p D_d$.
The conditions for the effective sources $f$ and $g$, which are necessary in order to correctly model the defect, will be characterized in the next subsection.

\subsection{Characterization of the effective sources}\label{subsec:char_eff}
Here we clarify the relation between the effective source pair $(f,g)$ and the density pair $(\varphi,\psi)$.

First we consider the integral equation for the solution $\widetilde{u}$ inside the central unit cell $Y$.
Inside $Y$, the solution $\tilde{u}$ can be decomposed as
\begin{align}\label{eq:u_rep_effective}
\tilde{u}
= 
\begin{cases}
H + \S_D^{k_w}[\psi] &\quad \mbox{in } Y\setminus \overline{D},
\\
\S_D^{k_b}[\varphi] &\quad \mbox{in } D,
\end{cases}
\end{align}
where $H$ satisfies the homogeneous equation $(\Delta + k_w^2) H = 0$ in $Y$ and the pair $(\varphi,\psi)\in L^2(\p D)^2$ satisfies
\begin{equation}\label{eq:AD}
\mathcal{A}_D
\begin{pmatrix}
\varphi 
\\
\psi
\end{pmatrix}
:=
\begin{pmatrix}
\S_D^{k_b} & - \S_D^{k_w}
\\[0.3em]
\ds \partial_\nu \S_D^{k_b} |_{-}
&
\ds -\delta \partial_\nu \S_D^{k_w} |_{+}
\end{pmatrix}
\begin{pmatrix}
\varphi 
\\
\psi
\end{pmatrix}
=
\begin{pmatrix}
H|_{\partial D}- f
\\[0.3em]
\ds\partial_\nu H |_{\partial D} - g
\end{pmatrix}.
\end{equation}

Next, let us consider the central cell $Y$ in the original problem \eqref{eq:scattering}. The central cell $Y$ contains a defect bubble $D_d$ and no sources.
Inside $Y$, the solution $u$ is represented as
\begin{align*}
u=
\begin{cases}
H + \S_{D_d}^{k_w}[\psi_d] &\quad \mbox{in } Y\setminus \overline{D_d},
\\
\S_{D_d}^{k_b}[\varphi_d] &\quad \mbox{in } D_d,
\end{cases}
\end{align*}
where
\begin{equation}\label{eq:ADd_original}
\A_{D_d}
\begin{pmatrix}
\varphi_d 
\\
\psi_d
\end{pmatrix}:=
\begin{pmatrix}
\S_{D_d}^{k_b} & - \S_{D_d}^{k_w}
\\[0.3em]
\ds \partial_\nu \S_{D_d}^{k_b} |_{-}
&
\ds -\delta \partial_\nu \S_{D_d}^{k_w} |_{+}
\end{pmatrix}
\begin{pmatrix}
\varphi_d 
\\
\psi_d
\end{pmatrix}
=
\begin{pmatrix}
H|_{\partial {D_d}}
\\[0.3em]
\ds\partial_\nu H |_{\partial {D_d}}
\end{pmatrix}.
\end{equation}

In order for the effective sources $f$ and $g$ to model the defect $D_d$ correctly, we should impose $u \equiv \widetilde{u} \   \mbox{in } (Y\setminus D_l) \cup D_s$, which implies $u \equiv \widetilde{u} \  \mbox{in } \mathbb{R}^2\setminus Y $ as well by an argument using analytic continuation.  In other words, the following conditions should be satisfied:
\begin{equation}\label{eq:SDdSD_inside}
\mathcal{S}_{D_d}^{k_b}[\varphi_d] \equiv \mathcal{S}_D^{k_b}[\varphi] \quad\mbox{in }D_s,
\end{equation}
and
\begin{equation}\label{eq:SDdSD_outside}
\mathcal{S}_{D_d}^{k_w}[\psi_d] \equiv \mathcal{S}_D^{k_w}[\psi] \quad\mbox{in }Y\setminus \overline{D_l}.
\end{equation}

Since $D$ and $D_d$ are circular disks, we can use a Fourier basis for functions on $\p D$ or $\p D_d$ in polar coordinates $(r,\theta)$ to make \eqref{eq:ADd_original}-\eqref{eq:SDdSD_outside} more explicit. Let us write $\varphi$ and $\psi$ in the form of Fourier series:
$$
\varphi(\theta) = \sum_{n\in\mathbb{Z}} \varphi_n e^{in \theta}, \quad
\psi(\theta) = \sum_{n\in\mathbb{Z}} \psi_n e^{in \theta}.
$$
Similarly, we also write $\varphi_d$ and $\psi_d$ as
$$
\varphi_d(\theta) = \sum_{n\in\mathbb{Z}} \varphi_{d,n} e^{in \theta}, \quad
\psi_d(\theta) = \sum_{n\in\mathbb{Z}} \psi_{d,n} e^{in \theta}.
$$
We define the subspace $V_{mn}$ of $L^2(\p D)^2$ as 
$$
V_{mn} : = \mbox{span}\{ e^{im \theta}\} \times \mbox{span}\{ e^{in \theta}\}, \quad
m,n\in \mathbb{Z}.
$$
Similarly, let ${V}_{mn}^d$ be the subspace of $L^2(\p D_d)^2$ with the same Fourier basis.

It is known that (see, for instance, \cite{bandgap})
\begin{equation}\label{eq:single_explicit}
\S_D^k[e^{i n \theta}]
=\frac{(-1)i\pi R}{2}\times \begin{cases}
J_n(k R) H_n^{(1)}(k r) e^{i n\theta}, &\quad r\geq R,
\\
H_n^{(1)}(k R) J_n(k r) e^{i n\theta}, &\quad 0\leq r < R. 
\end{cases}
\end{equation}
Therefore, the operator $\A_{D}$ in \eqref{eq:AD} has the following matrix representation as an operator from $V_{mn}$ to $V_{m'n'}$:
\begin{equation}\label{eq:AD_multipole}
(\A_{D})_{V_{mn}\rightarrow V_{m'n'}} = \delta_{mn}\delta_{m'n'} \frac{(-i)\pi R}{2} \begin{pmatrix}
  J_n(k_b R)H_n^{(1)}(k_b R) & -  J_n(k_w R)H_n^{(1)}(k_w R)
 \\
 k_b J_n'(k_b R) H_n^{(1)}(k_b R) & - \delta k_w J_n(k_w R) \big(H_n^{(1)}\big)'(k_w R)
 \end{pmatrix}.
\end{equation}
Similarly, the operator $\A_{D_d}$ in \eqref{eq:ADd_original}, as a mapping from $V_{mn}^d$ to $V_{m'n'}^d$, is represented as follows:
\begin{equation}\label{eq:ADd_multipole}
(\A_{D_d})_{V_{mn}^d\rightarrow V_{m'n'}^d} = \delta_{mn}\delta_{m'n'} \frac{(-i)\pi R_d}{2} \begin{pmatrix}
  J_n(k_b R_d)H_n^{(1)}(k_b R_d) & -  J_n(k_w R_d)H_n^{(1)}(k_w R_d)
 \\
 k_b J_n'(k_b R_d) H_n^{(1)}(k_b R_d) & - \delta k_w J_n(k_w R_d) \big(H_n^{(1)}\big)'(k_w R_d)
 \end{pmatrix}.
\end{equation}

Now, we consider \eqref{eq:SDdSD_inside} and \eqref{eq:SDdSD_outside}. We have from  \eqref{eq:single_explicit} that, inside $D_s$, 
\begin{align*}
\mathcal{S}_{D}^{k_b}[\varphi] &= \sum_{n\in\mathbb{Z}}\varphi_{n} \mathcal{S}_{D}^{k_b}[e^{i n \theta}]
= \frac{(-i)\pi R}{2}\sum_{n\in\mathbb{Z}}\varphi_{n}  H_n^{(1)}(k_b R) J_n(k_b r) e^{i n\theta}
\\
&= \frac{(-i)\pi R}{2}\sum_{n\in\mathbb{Z}}\varphi_{n}  \frac{H_n^{(1)}(k_b R)}{H_n^{(1)}(k_b R_d)}  H_n^{(1)}(k_b R) J_n(k_b r) e^{i n\theta}
\\
&= \mathcal{S}_{D_d}^{k_b}\Big[\sum_{n\in\mathbb{Z}}\varphi_{n} \frac{R}{R_d}\frac{H_n^{(1)}(k_b R)}{H_n^{(1)}(k_b R_d)}e^{in\theta}\Big].
\end{align*}
Similarly, outside $D_l$,
$$
\mathcal{S}_{D}^{k_w}[\psi] = \mathcal{S}_{D_d}^{k_w}\Big[\sum_{n\in\mathbb{Z}}\psi_{n} \frac{R}{R_d} \frac{J_n(k_w R)}{J_n(k_w R_d)}e^{in\theta}\Big].
$$
So, from \eqref{eq:SDdSD_inside} and \eqref{eq:SDdSD_outside}, we see that  
$$
\begin{pmatrix}
\varphi_d
\\[0.3em]
\ds\psi_d
\end{pmatrix}
=
\mathcal{P}_1
\begin{pmatrix}
\varphi
\\[0.3em]
\psi
\end{pmatrix},
$$
%\textbf{*** Do you want $\epsilon$ on this $\mathcal{P}_1$ as it is missing from the following $\mathcal{P}_1$'s ***}
where the operator $\mathcal{P}_1: L^2(\p D)^2\rightarrow L^2(\p D_d)^2$ is given by 
$$
(\mathcal{P}_1)_{V_{mn}\rightarrow V_{m'n'}^d} = \delta_{mn}   \delta_{m'n'}\frac{R}{R_d}
\begin{pmatrix}
\ds\frac{H_n^{(1)}(k_b R) }{H_n^{(1)}(k_b R_d)}
& 0
\\
0& \ds\frac{J_n(k_w R) }{J_n(k_w R_d)}
\end{pmatrix}.
\quad
$$
In the same way, we can obtain that
$$
\begin{pmatrix}
H|_{\partial D_d}
\\[0.3em]
\ds\partial_\nu H |_{\partial D_d}
\end{pmatrix}
=
\mathcal{P}_2
\begin{pmatrix}
H|_{\partial D}
\\[0.3em]
\ds\partial_\nu H |_{\partial D}
\end{pmatrix},
$$ 
where the operator $\mathcal{P}_2: L^2(\p D)^2\rightarrow L^2(\p D_d)^2 $ is given by
$$
(\mathcal{P}_2)_{V_n\rightarrow V_m^d} = \delta_{mn}  
\begin{pmatrix}
\ds \frac{J_n(k_w R_d) }{J_n(k_w R)}
& 0
\\
0& \ds\frac{J_n'(k_w R_d) }{J_n'(k_w R)}
\end{pmatrix}.
$$
Therefore, using \eqref{eq:ADd_original}, we arrive at
\begin{equation}\label{eq:ADd}
\mathcal{A}_D^\epsilon \begin{pmatrix}
\varphi
\\[0.3em]
\psi
\end{pmatrix} := (\mathcal{P}_2)^{-1}\mathcal{A}_{D_d} \mathcal{P}_1 
\begin{pmatrix}
\varphi
\\[0.3em]
\psi
\end{pmatrix}
=\begin{pmatrix}
H|_{\partial D}
\\[0.3em]
\ds\partial_\nu H |_{\partial D}
\end{pmatrix}.
\end{equation}
Thus, \eqref{eq:AD} yields
\begin{equation}\label{eq:relation_density_source}
(\mathcal{A}_D^\epsilon - \mathcal{A}_D)\begin{pmatrix}
\varphi
\\[0.3em]
\psi
\end{pmatrix} = 
\begin{pmatrix}
f
\\[0.3em]
g
\end{pmatrix}.
\end{equation}

We have obtained an explicit relation between the pair $(\varphi,\psi)$ and the effective sources $(f,g)$. 
If the effective sources $f$ and $g$ satisfy \eqref{eq:relation_density_source}, then the pair $(f,g)$ will result in the generation of the same scattered field as the one induced by the defect bubble $D_d$. In other words, $\widetilde{u}\equiv u$ outside $D_l\setminus \overline{D_s}$. In what follows, for convenience of notation, we will identify the solution $\widetilde u$ with the original one $u$.

\subsection{Floquet transform of the solution}

Here we derive an integral equation for the effective source problem \eqref{eq:scattering_fictitious}.
We apply the Floquet transform to the solution $u$ with the quasi-periodic parameter $\alpha$ as follows
$$
u^\alpha = \sum_{m\in \Z^2} u(x-m) e^{i\alpha \cdot m}.
$$
The transformed solution $u^\alpha$ satisfies
\begin{equation} \label{eq:scattering_quasi}
\left\{
\begin{array} {ll}
	&\ds \nabla \cdot \frac{1}{\rho_w} \nabla  u^\alpha+ \frac{\omega^2}{\kappa_w} u^\alpha  = 0 \quad \text{in} \quad Y \setminus \overline{D}, \\
	\nm
	&\ds \nabla \cdot \frac{1}{\rho_b} \nabla  u^\alpha+ \frac{\omega^2}{\kappa_b} u^\alpha  = 0 \quad \text{in} \quad D, \\
	\nm
	&\ds  u^\alpha|_{+} -u^\alpha|_{-}  =f   \quad \text{on} \quad \partial D, \\
	\nm
	& \ds  \frac{1}{\rho_w} \frac{\partial u^\alpha}{\partial \nu} \bigg|_{+} - \frac{1}{\rho_b} \frac{\partial u^\alpha}{\partial \nu} \bigg|_{-} =g \quad \text{on} \quad \partial D,
	\\
	\nm
	& \ds e^{-i \alpha \cdot x} u^\alpha \text{ is periodic}.
\end{array}
\right.
\end{equation} 
The solution $u^\alpha$ can be represented using quasi-periodic layer potentials as
\begin{align*}
u^\alpha
= 
\begin{cases}
\S_D^{\alpha,k_w}[\psi_\alpha], &\quad \mbox{in } Y\setminus \overline{D},
\\
\S_D^{k_b}[\varphi_\alpha], &\quad \mbox{in } D,
\end{cases}
\end{align*}
where
 the pair $(\varphi^\alpha,\psi^\alpha)\in L^2(\p D)^2$ is the solution to
\begin{equation}\label{eq:phipsialpha}
\mathcal{A}^\alpha(\omega,\delta)
\begin{pmatrix}
\varphi^\alpha 
\\[0.3em]
\psi^\alpha
\end{pmatrix}
:=
\begin{pmatrix}
\S_D^{k_b} & - \S_D^{\alpha,k_w}
\\[0.3em]
\ds \partial_\nu \S_D^{k_b} |_{-}
&
\ds -\delta \partial_\nu \S_D^{\alpha,k_w} |_{+}
\end{pmatrix}
\begin{pmatrix}
\varphi^\alpha 
\\[0.3em]
\psi^\alpha
\end{pmatrix}
=
\begin{pmatrix}
- f
\\[0.3em]
- g
\end{pmatrix}.
\end{equation}
Since the operator $\mathcal{A}_\alpha$ is invertible for $\omega$ in the bandgap \cite{bandgap}, 
we have
$$
\begin{pmatrix}
\varphi^\alpha 
\\[0.3em]
\psi^\alpha
\end{pmatrix} 
= \mathcal{A}^\alpha(\omega,\delta)^{-1}\begin{pmatrix}
- f
\\[0.3em]
- g
\end{pmatrix}.
$$

The original solution $u$ can be recovered by the inversion formula as follows
\begin{equation}
u(x)=\frac{1}{(2\pi)^2}\int_{BZ} u^\alpha(x) d\alpha.
\end{equation}
%\textbf{*** Often this has a normalization factor of something like $1/(2\pi)^2$ in front. Not sure if it is missing intentionally here or not. ***}
Then, inside the region $D$, the solution $u$ satisfies
$$
u= \S_D^{k_b}\big[\frac{1}{(2\pi)^2}\int_{BZ}\varphi^\alpha d\alpha \big].
$$
Similarly, inside the region $Y\setminus \overline{D}$, we have
\begin{align*}
\ds
u&=\frac{1}{(2\pi)^2}\int_{BZ}\S_D^{\alpha,k_w}[\psi^\alpha] d\alpha
=\S_D^{k_w}\big[\frac{1}{(2\pi)^2}\int_{BZ}\psi^\alpha  d\alpha\big] + \frac{1}{(2\pi)^2}\int_{BZ} \sum_{m\in\mathbb{Z}^2,m\neq 0}\S_D^{k_w}[\psi^\alpha](\,\cdot - m)e^{im\cdot \alpha} d\alpha.
\end{align*}
Note that the second term in the right-hand side satisfies the homogeneous Helmholtz equation $(\Delta + k_w^2 )u=0$ in $Y\setminus \overline{D}$. 
So, in view of \eqref{eq:u_rep_effective}, we can identify $\varphi,\psi$ and $H$ as follows:
$$
\varphi = \frac{1}{(2\pi)^2}\int_{BZ} \varphi^\alpha d\alpha,
\quad
\psi = \frac{1}{(2\pi)^2}\int_{BZ} \psi^\alpha d \alpha,
$$
and
$$
H=\frac{1}{(2\pi)^2}\int_{BZ} \sum_{m\in\mathbb{Z}^2,m\neq 0}\S_D^{k_w}[\psi^\alpha](\,\cdot - m)e^{im\cdot \alpha} d\alpha.
$$
Therefore, from \eqref{eq:phipsialpha}, we get the following result.
\begin{prop} The density pair $(\varphi,\psi)$ and the effective source pair $(f,g)$ satisfy
\begin{equation}\label{eq:IntAa}
\begin{pmatrix}
\varphi
\\[0.3em]
\psi 
\end{pmatrix}
=\left(\frac{1}{(2\pi)^2}\int_{BZ} \mathcal{A}^\alpha(\omega,\delta)^{-1} d\alpha \right)\begin{pmatrix}
-f
\\[0.3em]
-g 
\end{pmatrix}
\end{equation}
for $\omega$ in the bandgap.
\end{prop}

\subsection{The integral equation for the effective sources}

Here we derive the integral equation for the effective source pair $(f,g)$. We have the following result.

\begin{prop}
The effective source pair $(f,g)\in L^2(\partial D)^2$ satisfies the following integral equation:
\begin{align}\label{eq:fg}
\mathcal{M}^\epsilon(\omega)\begin{pmatrix}
f
\\
g
\end{pmatrix}:=\bigg(I+(\A_{D}^\epsilon(\omega,\delta) - \A_{D}(\omega,\delta))\frac{1}{(2\pi)^2}\int_{BZ}\A^\alpha(\omega,\delta)^{-1} d\alpha\bigg)
\begin{pmatrix}
f 
\\[0.3em]
 g
\end{pmatrix}
=
\begin{pmatrix}
0
\\[0.3em]
0
\end{pmatrix},
\end{align}
for $\omega$ in the bandgap.
\end{prop}

\begin{proof}
Equation \eqnref{eq:fg} is an immediate consequence of equations \eqnref{eq:IntAa} and \eqnref{eq:relation_density_source}.
\end{proof}

Thus, if we find a value of $\omega$ in the bandgap such that there exists a non-trivial solution pair $(f,g)$ to (\ref{eq:fg}), then we will find a resonant frequency for the localized mode. 

\section{Subwavelength localized modes} \label{sec-3}

\subsection{The resonant frequency of the localized mode}

Here we prove that a frequency for the localized mode exists slightly above $\max_\alpha \omega_1^\alpha$.
In what follows, let us omit the subscript in $\omega_1^\alpha$ for ease of notation. We also do not make explicit the dependence on $\delta$.

We need to study the characteristic value of the operator $\mathcal{M}^\epsilon$ appearing in \eqref{eq:fg}. Let us first analyse the operator $\int_{BZ} (\A^{\alpha})^{-1}d\alpha$.
Since $\omega^\alpha$ is a simple pole of the mapping $\omega \mapsto \A^\alpha(\omega,)^{-1}$ in a neighbourhood of $\omega^\alpha$,
according to \cite{lecturenotes}, we can write
\begin{equation}
\A^\alpha(\omega)^{-1} = \frac{\L^\alpha}{\omega- \omega^\alpha} + \mathcal{R}^\alpha(\omega),
\end{equation}
where the operator-valued function $\mathcal{R}^\alpha(\omega)$ is holomorphic in a neighbourhood of $\omega^\alpha$, and the operator $\L^\alpha$ maps $L^2(\p D)^2$ onto $\ker \A^\alpha(\omega^\alpha,\delta)$. 
Let us write
$$
\ker {\mathcal{A}^\alpha(\omega^\alpha)} = \mbox{span} \{\Psi^\alpha\}, \quad \ker {\big(\mathcal{A}^\alpha(\omega^\alpha)\big)^\dagger} = \mbox{span} \{\Phi^\alpha\},
$$
where $^\dagger$ denotes the adjoint operator. Then, as in \cite{lecturenotes, thinlayer}, it can be shown that 
$$
\L^\alpha = \frac{\langle \Phi^\alpha,\ \cdot\ \rangle \Psi^\alpha}{\langle \Phi^\alpha, \frac{d}{d \omega}\A^\alpha\big|_{\omega=\omega^\alpha} \Psi^\alpha \rangle},
$$
where $\langle \,\cdot \,,\,\cdot\, \rangle$ stands for the standard inner product of $L^2(\partial D)^2$.

Hence the operator $\mathcal{M}^\epsilon$  can be decomposed as
$$
\mathcal{M}^\epsilon = I + (\A_{D}^\epsilon - \A_D) \frac{1}{(2\pi)^2}\int_{BZ}\frac{\L^\alpha}{\omega-\omega_\alpha}d\alpha +  (\A_{D}^\epsilon - \A_D) \frac{1}{(2\pi)^2}\int_{BZ} \mathcal{R}_\alpha d\alpha.
$$
Note that the third term in the right-hand side is holomorphic with respect to $\omega$.

Denote by $\alpha^* = (\pi,\pi)$ and $\omega^* = \omega^{(\pi,\pi)}$.
In \cite{highfrequency}, it was proved  that, using the symmetry of the square array of bubbles, $\omega^\alpha$ attains its maximum at $\alpha=\alpha^*$.  
Since we are assuming that each bubble is a circular disk, we can derive a slightly more refined result as shown in the following lemma. 
\begin{lem} \label{lem:max_omegaalp}
The characteristic value $\omega_\alpha$ attains its maximum at $\alpha=\alpha^*$. Moreover, for $\alpha$ near $\alpha_*$, we have
$$
\omega^\alpha = \omega^* -   \frac{1}{2}c_\delta |\alpha-\alpha_*|^2   + o(|\alpha-\alpha_*|^2).
$$
Here, $c_\delta$ is a positive constant depending on $\delta$ and scales as $c_\delta = O(\sqrt{\delta})$.
\end{lem}

In view of the above lemma, we can expect that the operator $\int_{BZ}\frac{\L^\alpha}{\omega-\omega^\alpha}$ becomes singular when $\omega\rightarrow \omega^*(=\max_{\alpha } \omega^{\alpha})$.
Let us extract its singular part explicitly. Before doing this, we introduce some notations. Denote by
$
\A^* = \A^{(\pi,\pi)},  \Phi^* = \Phi^{(\pi,\pi)},$ and $\L^* = \L^{(\pi,\pi)}
$.
We also define a small neighbourhood $V$ of $\omega^*$ which excludes the real interval $(-\infty,\omega^*]$, namely,
$$
V= \{|\omega-\omega^*|<r_*\} \setminus (-\infty,\omega^*],
$$
for some small enough $r_*>0$.
By Lemma \ref{lem:max_omegaalp}, we have
\begin{align*}
\frac{1}{(2\pi)^2}\int_{BZ}\frac{\L^\alpha}{\omega-\omega^\alpha} d\alpha &= \L^*\frac{1}{(2\pi)^2}\int_{BZ} \frac{1}{\omega - \omega^*+\frac{1}{2} c_\delta |\alpha-\alpha^*|^2} d\alpha + B_1(\omega)
\\
&=\L^*\frac{1}{(2\pi)^2}\int_{|\alpha-\alpha^*|<1} \frac{}{\omega - \omega^*+\frac{1}{2} c_\delta |\alpha-\alpha^*|^2} d\alpha + B_2(\omega).
\end{align*}
Hereafter, $B_j$ means a bounded function with respect to $\omega$ in $V$.
Then, using the polar coordinates $\alpha-\alpha_*=(r',\theta')$, we get
\begin{align*}
\frac{1}{(2\pi)^2}\int_{BZ}\frac{\L^\alpha}{\omega-\omega^\alpha} d\alpha
&=\frac{1}{(2\pi)^2}\int_0^{1} \frac{2\pi r'}{\omega-\omega_* + \frac{1}{2} c_\delta r'^2} dr' +B_2(\omega)
\\
&= -\frac{1}{2\pi c_\delta}\ln (\omega-\omega_*) + B_3(\omega).
\end{align*}
Here, the usual principal branch is taken for the logarithm  and so the operator $\int_{BZ}\frac{\L^\alpha}{\omega-\omega^\alpha} d\alpha$ has a branch cut on $(-\infty,\omega^*]$.
We also observe that, for $\omega \in V$,
$$
\ln(\omega-\omega^*) (\A_{D}^\epsilon(\omega)-\A_{D}(\omega))=
\ln(\omega-\omega^*) (\A_{D}^\epsilon(\omega^*)-\A_{D}(\omega)^*) + B_4(\omega).
$$

Therefore, the integral equation \eqref{eq:fg} can be rewritten as
\begin{align*}
\mathcal{M}^\epsilon(\omega)\begin{pmatrix}
f
\\
g
\end{pmatrix}= \Big(I - \frac{\ln(\omega-\omega^*)}{2\pi c_\delta} (\A_{D}^\epsilon(\omega^*)-\A_{D}(\omega^*))\mathcal{L}^*  + {\mathcal{R}^\epsilon(\omega)}
\Big)
\begin{pmatrix}
f
\\
g
\end{pmatrix}
=
\begin{pmatrix}
0
\\
0
\end{pmatrix},
\end{align*}
where $\mathcal{R}^\epsilon$ is analytic and bounded with respect to $\omega$ in $V$ and satisfies
$
{\mathcal{R}}^\epsilon = O(\epsilon)
$.
We first consider the principal part $\mathcal{N}^\epsilon(\omega)$, namely,
$$
\mathcal{N}^\epsilon : \omega \mapsto \mathcal{N}^\epsilon(\omega)=
I - \frac{1}{2\pi c_\delta}\ln(\omega-\omega^*) (\A_{D}^\epsilon(\omega^*)-\A_{D}(\omega^*))\mathcal{L}^*. 
$$
Let us find its characteristic value $\widehat\omega$, {\it i.e.}, $\widehat\omega\in V$ such that there exists a non-trivial function $\widehat\Phi$ satisfying $\mathcal{N}^\epsilon(\widehat\omega) \widehat\Phi =0$. Equivalently, we have
\begin{align*}
\widehat\Phi - \frac{1}{2\pi c_\delta}\ln(\widehat\omega-\omega^*) \frac{ (\A_{D}^\epsilon-\A_{D})(\omega^*) \Psi^*}{\langle \Phi^*, \frac{d}{d \omega}\A^*\big|_{\omega=\omega^*} \Psi^* \rangle}\langle \Phi^*, \widehat\Phi \rangle =0.
\end{align*}
Then, by multiplying by $\Phi^*$,
\begin{align*}
\langle \Phi^*, \widehat\Phi\rangle \bigg( 1 -  \frac{1}{2\pi c_\delta}\ln(\widehat\omega-\omega^*) \frac{ \langle \Phi^*,(\A_{D}^\epsilon-\A_{D})(\omega^*) \Psi^*\rangle}{\langle \Phi^*, \frac{d}{d \omega}\A^*\big|_{\omega=\omega^*} \Psi^* \rangle} \bigg)=0.
\end{align*}
Since $\langle \Phi^*, \widehat\Phi \rangle=0$ would imply $\widehat\Phi=0$, we get
\begin{equation}\label{eq:temptemp}
1 -  \frac{1}{2\pi c_\delta}\ln(\widehat\omega-\omega^*) \frac{ \langle \Phi^*,(\A_{D}^\epsilon-\A_{D})(\omega^*) \Psi^*\rangle}{\langle \Phi^*, \frac{d}{d \omega}\A^*\big|_{\omega=\omega^*} \Psi^* \rangle} = 0.
\end{equation}

Before solving the above equation for $\omega$, we need the following Lemma whose proof will be given in Subsection \ref{subsec:proofs_estim}.

\begin{lem}\label{lem:estim1} The following results hold:
	\begin{itemize}
		\item[(i)]When $\delta\rightarrow 0$, we have
		$$
		\big\langle \Phi^*, \frac{d}{d \omega}\A^*(\omega^*,\delta) \Psi^* \big\rangle = -2\pi{\omega^*\ln \omega^*} R^3 + O(\sqrt{\delta}),
		$$
		which is positive for $\delta$ small enough.
		\item[(ii)] For a fixed $\epsilon$, when $\delta\rightarrow 0$ we have
		\begin{align*}
		\left\langle\Phi^*,\big(\A_{D}^\epsilon(\omega^*,\delta)-\A_{D}(\omega^*,\delta)\big){ \Psi^*} \right\rangle &=\delta\epsilon\ln\omega^*\left(R\|\psi_{\alpha^*}\|^2_{L^2(\p D)}  - 2\capacity_{D,\alpha^*} \right) + O(\epsilon\delta + \epsilon^2\delta\ln\delta),\\
		& := \epsilon\big(\delta\ln\omega^*S(R)\big) + O(\epsilon\delta + \epsilon^2\delta\ln\delta),
		\end{align*}
		where $\psi_{\alpha^*} = (\S_{D}^{\alpha^*,0})^{-1}[\chi_{\p D}]$ and $\capacity_{D,\alpha^*} = -\langle \psi_{\alpha^*}, \chi_{\p D}\rangle$. For small $\epsilon$ and $\delta$, the sign of $\delta\ln\omega^*S(R)$ varies with $R$ as follows: $\delta\ln\omega^*S(R) > 0$ for small enough $R$, while $\delta\ln\omega^*S(R) < 0$ for $R$ close enough to $1/2$.
	\end{itemize}
\end{lem}
In view of Lemma \ref{lem:estim1}, we have two different regimes: The dilute regime when $R$ is small and the non-dilute regime when $R$ is close to $1/2$. For equation \eqref{eq:temptemp} to have a real solution $\widehat\omega$ close to $\omega_*$, we need $\epsilon < 0$ in the dilute regime and $\epsilon > 0$ in the non-dilute regime. Under this assumption we have
\begin{align*}
\widehat\omega-\omega_* = 
  \exp\big(2\pi{c_\delta} \big\langle \Phi^*, \frac{d}{d \omega}\A^*(\omega^*) \Psi^* \big\rangle \big\langle\Phi^*,\big(\A_{D}^\epsilon(\omega^*)-\A_{D}(\omega^*)\big){ \Psi^*} \big\rangle^{-1}\big).
\end{align*}
We can also see that the right-hand side is positive and goes to zero as $\epsilon$ tends to zero. In other words, $\widehat\omega\rightarrow \omega^*$ as $\epsilon\rightarrow 0$. Now we turn to the full operator $\mathcal{M}^\epsilon$. Recall that $\mathcal{M}^\epsilon = \mathcal{N}^\epsilon + \mathcal{R}^\epsilon$and $\mathcal{R}^\epsilon$ is holomorphic, bounded and satisfies $\mathcal{R}^\epsilon=O(\epsilon)$ in $V$. So by Gohberg-Sigal theory \cite{lecturenotes}, we can conclude that there exists a characteristic value $\omega^\epsilon$ of the operator-valued function $\mathcal{M}^\epsilon$ near $\omega^*$. Let us denote its associated root function by $\Phi^\epsilon$. We choose $\Phi^\epsilon$ so that $\langle \Phi^*, \Phi^\epsilon\rangle=1$.  Then, as in the derivations of \eqref{eq:temptemp}, we can obtain
\begin{equation*}
1 - \frac{1}{2\pi c_\delta}\ln(\omega^\epsilon-\omega^*) \mu(\epsilon,\delta,R)  + \langle \Phi^*,\mathcal{R}^\epsilon(\omega^\epsilon) \Phi^\epsilon\rangle= 0,
\end{equation*}
where
$$
\mu(\epsilon,\delta,R) = \frac{ \langle \Phi^*,(\A_{D}^\epsilon-\A_{D})(\omega^*) \Psi^*\rangle}{\langle \Phi^*, \frac{d}{d \omega}\A^*\big|_{\omega=\omega^*} \Psi^* \rangle}. 
$$
Note that $\langle \Phi^*,\mathcal{R}^\epsilon(\omega) \Phi^\epsilon\rangle = O(\epsilon)$ for small $\epsilon$. 
We see from Lemma \ref{lem:estim1} that
$$
{\mu}(\epsilon,\delta,R) =  \frac{\delta\epsilon\left(R\|\psi_{\alpha^*}\|^2_{L^2(\p D)}  - 2\capacity_{D,\alpha^*} \right)}{-2\pi\omega^* R^3} + O\Big(\epsilon\frac{\sqrt{\delta}}{\ln\delta} + \sqrt{\delta}\epsilon^2 \Big),
$$
as $\delta\rightarrow 0$ and $\epsilon \rightarrow 0$, and the leading order term is negative.
We remark that the frequency $\omega^\epsilon$ is a real number since the operator in the defect problem \eqref{eq:scattering} can be considered as a real compact perturbation of a self-adjoint periodic operator \cite{figotin2, MMMP4}. 

Therefore, we have proved the following theorem which is the main result of this paper.
\begin{thm} 
	Assume that $\delta$ is small enough and the pair $(R,\epsilon)$ satisfies one of the two assumptions
	\begin{itemize}
		\item[(i)] $R$ small enough and $\epsilon<0$ small enough in magnitude (Dilute regime),
		\item[(ii)] $R$ close enough to $1/2$ and $\epsilon>0$ small enough (Non-dilute regime).
	\end{itemize}
	Then there exists one frequency value $\omega^\epsilon$ such that the problem \eqref{eq:scattering} has a non-trivial solution. Moreover, in both cases, $\omega^\epsilon$ is slightly above $\omega^*$ and we have
\begin{align*}
\omega^\epsilon-\omega_* = 
  \exp\left(-\frac{4\pi^2c_\delta\omega^* R^3}{\delta\epsilon\left(R\|\psi_{\alpha^*}\|^2_{L^2(\p D)}  - 2\capacity_{D,\alpha^*} \right)} +O\left(\frac{1}{\epsilon\ln\delta} + 1\right)\right),
\end{align*}
when $\epsilon$ and $\delta$ goes to zero.

\end{thm}

\begin{rmk}
Since $\omega^\epsilon$ is slightly above $\omega^*=\max_{\alpha}\omega^\alpha$, we have that $\omega^\epsilon$ is located in the bandgap region. This means that the corresponding function $u^\epsilon$, the solution to \eqref{eq:scattering}, should be localized around the defect.
\end{rmk}

\begin{rmk}
We have shown that it is enough to perturb the size of the defect slightly to create the localized mode in the subwavelength case.
This is different from the photonic crystal case where the defect should be large enough to ensure the existence of the localized mode.
\end{rmk}

\begin{rmk}
There are different physical mechanisms at work in the two different regimes. In the dilute regime, the interactions are weak between the resonators. Decreasing the defect bubble size increases the Minnaert resonance of that single resonator \cite{first}, creating a mode inside the bandgap. In the non-dilute regime, the bubbles are closely separated and have strong interactions. Increasing the defect bubble size reduces the separation, which increases the resonance frequency of that system of bubbles \cite{doublenegative}. As we will numerically verify in Section \ref{sec-4}, the shift occurs at $R=1/3$, which corresponds to the case where the bubble radius and the bubble separation are equal.
\end{rmk}

\subsection{Proof of Lemma \ref{lem:estim1} }\label{subsec:proofs_estim}
\noindent \textbf{Computation of $\Phi^*$ and $\Psi^*$}.
Recall that $\ker \mathcal{A}^*(\omega^*,\delta)$ is spanned by $\Psi^*$ and $\ker {\big(\mathcal{A}^*(\omega^*,\delta)\big)^\dagger}$ is spanned by $\Phi^*$. Let us write $$
\Psi^{*} = \begin{pmatrix}
\psi_1^*\\ \psi_2^*
\end{pmatrix}, \quad \Phi^{*} = \begin{pmatrix}
\varphi_1^*\\ \varphi_2^*
\end{pmatrix},
$$
and assume that $\Psi^*$ and $\Phi^*$ are chosen so that $\int_{\p D}\psi_1^* = 2\pi R$ and $\int_{\p D}\phi_2^* = 2\pi R$. 
Next we consider the kernel of the limiting operator of $\mathcal{A}^*(\omega,\delta)$ when $\omega\rightarrow0 $, $ \delta\rightarrow0$. Recall that $\S_D^\omega$ has a $\log$-singularity as $\omega\rightarrow 0$, given by \eqnref{eq:Sexpansion}. Define
$$
\hat{\mathcal{A}}^*(\omega^*)=\begin{pmatrix}
\hat{\mathcal{S}}^{\omega^*}_D & - \mathcal{S}_D^{\alpha^*,0} 
\\
-\frac{1}{2}I + \mathcal{K}_D^* & 0
\end{pmatrix}.
$$
It is known that $\ker(-\frac{1}{2}I+\K_D^*)$ is one-dimensional (see, for instance, \cite{Ammari2007_polarizationBook}). We choose an element $\psi_1^{0,*}\in L^2(\p D)$ such that
\begin{align*}
\ds\bigg(-\frac{1}{2}I+\K_D^*\bigg)[\psi_1^{0,*}] = 0,
\quad
\ds\int_{\p D} \psi_1^{0,*} =2 \pi R.
\end{align*}
Then we define
$
\ds\psi_2^{0,*} = (\mathcal{S}_D^{\alpha^*,0})^{-1}[\hat{\S}_D^{\omega^*}\big[\psi_1^{0,*}]\big]
$ and denote $$
\Psi^{0,*} = \begin{pmatrix}
\psi_1^{0,*}\\ \psi_2^{0,*}
\end{pmatrix}.
$$
It can be shown that $\Psi^* = \Psi^{0,*} + O(\delta\ln\delta)$ as $\delta\rightarrow 0$. To compute $\Phi^*$, we use the well-known fact that the kernel $\ker(-\frac{1}{2}I+\K_D)$ is spanned by the constant function $\chi_{\p D}$. Denote $$
\Phi^{0,*} = \begin{pmatrix}
0\\ \chi_{\p D}
\end{pmatrix}.
$$
Then it can be shown that $\Phi^* =\Phi^{0,*} + O(\delta) $ as $\delta\rightarrow 0$.

Since $D$ is a disk we can compute $\psi_1^{0,*}$ and $\psi_2^{0,*}$ explicitly. Since $\int_{\p D} \psi_1^{0,*}=2\pi R$, we have $\psi_1^{0,*}= \chi_{\p D}$. Moreover, we have
\begin{align}
\hat{\S}^{\omega^*}_{D}[\chi_{\p D}] &= \S_D[\chi_{\p D}] + 2\pi R \eta_{\omega^*} \chi_{\p D} \nonumber \\ 
&= \left( R\ln R + 2\pi R \eta_{\omega^*} \right) \chi_{\p D} \nonumber \\ 
&= R \ln\omega^*\chi_{\p D} + O(1). \label{eq:Schi}
\end{align}  
Define
$$\psi_{\alpha^*} = (\S_D^{\alpha^*,0})^{-1}[\chi_{\p D}].$$
Then we get $\psi_2^{0,*} =  R\ln\omega^*\psi_{\alpha^*} + O(1)$ as $\delta\rightarrow 0$.  Expand $\psi_2^{0,*}$ in the Fourier basis as
$$ \psi_2 = \sum_{m\in \Z} a_m e^{im\theta}.$$
We will need the first coefficient $a_0$. Using \eqnref{eq:Schi} it follows that 
\begin{align*}
a_0 &= \frac{\ln\omega^*}{2\pi}\int_{\partial D} \left( \S_{D}^{\alpha^*}\right) ^{-1} [\chi_{\p D}] \dx \sigma + O(1) \\
&= -\frac{\ln\omega^*}{2\pi}\capacity_{D,\alpha^*}  + O(1).
\end{align*}
Finally, using the method from \cite{thinlayer} section 4.3, we can derive the following improved formula for $\Phi^*$ as $\delta\rightarrow 0$:
$$\Phi^* = \begin{pmatrix}
0\\ \chi_{\p D}
\end{pmatrix} - \delta \begin{pmatrix}
\psi_{\alpha^*} \\ 0
\end{pmatrix} + O(\delta^2\ln\delta). $$

\bigskip
\noindent \textbf{Proof of (i)}.
We need the low frequency asymptotics of the operator $\A^*(\omega^*,\delta)$. We use the following asymptotic expansions of the Hankel function for small arguments:
\begin{align*}
\big(H_0^{(1)}\big)'(z) &= \frac{2i}{\pi }\frac{1}{z} - \frac{i}{\pi} z\ln z +O(z),
\\
\big(H_0^{(1)}\big)''(z) &= -\frac{2i}{\pi }\frac{1}{z^2} - \frac{i}{\pi} \ln z +O(1).
\end{align*}
Straightforward computations show that, for small $k$,
\begin{align*}
\frac{d}{dk}\S_{D}^k[\phi] = \frac{1}{2\pi k} \int_{\p D} \phi(y) d\sigma(y) -\frac{k\ln k}{4\pi}   \int_{\p D} |x-y|^2 \phi(y) d\sigma(y) + O(k),
\end{align*}
and
\begin{align*}
\frac{d}{dk}(\K_{D}^k)^*[\phi] &= -\frac{k\ln k}{2\pi}   \int_{\p D} \langle x-y,\nu_x\rangle \phi(y) d\sigma(y) + O(k).
\end{align*}
Using low-frequency asymptotics of the quasi-periodic Green's function $\Gamma^{\alpha,k}$ \cite{lecturenotes}, it follows that 
\begin{align*}
\frac{d}{dk}\S_{D}^{\alpha^*,k}[\phi] = O(k), \quad \mathrm{and} \quad \frac{d}{dk}\left(\K_{D}^{\alpha^*,k}\right)^*[\phi] = O(1).
\end{align*}
Using $\omega = O(\sqrt{\delta})$ we find, on the subspace $V_{00}$,
$$ \left( \frac{d}{d \omega}\A^*(\omega^*,\delta) \Psi^*\right)_{V_{00}} = \begin{pmatrix} O(\frac{1}{\sqrt{\delta}}) \\ \frac{d}{dk}(\K_{D}^k)^*[\chi_{\p D}] + O(\delta\ln\delta)\end{pmatrix}_{V_{00}}. $$
In total, as $\delta\rightarrow 0$ we have
\begin{align*}
\big \langle \Phi^*, \frac{d}{d \omega}\A^*(\omega^*,\delta) \Psi^* \big \rangle
&= \langle \chi_{\p D},\frac{d}{d\omega}(\K_{D}^{\omega})^*[\chi_{\p D}]\rangle + O(\sqrt{\delta})
\\    
&= -\frac{\omega^*\ln \omega^*}{2\pi}  \int_{\p D} \int_{\p D} \langle x-y,\nu_x\rangle  d\sigma(y) d\sigma(x) 
=-2\pi{\omega^*\ln \omega^*} R^3 + O(\sqrt{\delta}).
\end{align*}
So, (i) is proved.
\qed
\bigskip

\noindent \textbf{Proof of (ii)}. 
Using equation \ref{eq:ADd}, we have
\begin{equation*}
%(\A_{D}^\epsilon)_{V_{mn}\rightarrow V_{m'n'}} = \delta_{mn}\delta_{m'n'} \frac{(-i)\pi R}{2} \begin{pmatrix} J_n(k_b R)H_n^{(1)}(k_b R) \frac{J_n(k_b R_d) }{J_n(k_w R_d)} & -  J_n(k_w R)H_n^{(1)}(k_w R_d) \frac{J_n(k_w R) }{J_n(k_w R_d)}\\k_b J_n'(k_w R) H_n^{(1)}(k_b R) \frac{J_n'(k_b R_d) }{J_n'(k_w R_d)} & - \delta k_w J_n(k_w R) \big(H_n^{(1)}\big)'(k_w R_d) \frac{J_n'(k_w R)}{J_n'(k_w R_d) }\end{pmatrix}.
(\A_{D}^\epsilon)_{V_{mn}\rightarrow V_{m'n'}} = \delta_{mn}\delta_{m'n'} \frac{(-i)\pi R}{2} \begin{pmatrix} J_n(\omega R)H_n^{(1)}(\omega R) & -  J_n(\omega R)H_n^{(1)}(\omega R_d) \frac{J_n(\omega R) }{J_n(\omega R_d)}\\\omega J_n'(\omega R) H_n^{(1)}(\omega R) & - \delta \omega J_n(\omega R) \big(H_n^{(1)}\big)'(\omega R_d) \frac{J_n'(\omega R)}{J_n'(\omega R_d) }\end{pmatrix}.
\end{equation*}	
Consequently, the operator $(\mathcal{A}_D^\epsilon - \mathcal{A}_D)$ is given by
\begin{equation*}
(\mathcal{A}_D^\epsilon - \mathcal{A}_D)_{V_{mn}\rightarrow V_{m'n'}} = \delta_{mn}\delta_{m'n'} \frac{(-i)\pi RJ_n(\omega R)}{2} \begin{pmatrix} 0 & H_n^{(1)}(\omega R) - \frac{J_n(\omega R)H_n^{(1)}(\omega R_d)  }{J_n(\omega R_d)}\\0 & \delta \omega \left(\big(H_n^{(1)}\big)'(\omega R) - \frac{J_n'(\omega R)\big(H_n^{(1)}\big)'(\omega R_d)}{J_n'(\omega R_d)} \right)\end{pmatrix}.
\end{equation*}	
Using asymptotic expansions of the Bessel function $J_n(z)$ and the Hankel function $H_n^{(1)}(z)$, for small $z$, straightforward computations show that 
\begin{align*}
\frac{(-i)\pi R}{2}\frac{J_n(\omega R)}{J_n(\omega R_d)}\left(H_n^{(1)}(\omega R){J_n(\omega R_d)} - J_n(\omega R)H_n^{(1)}(\omega R_d) \right) &= \begin{cases} R\ln\frac{R}{R_d} + O(\omega\ln\omega) \quad &n = 0 \\ -\frac{R}{2n}\left(1-\frac{R^{2n}}{R_d^{2n}}\right) + O(\omega) &n\neq 0\end{cases} \\
&= -\epsilon + O(\epsilon^2 + \omega\ln\omega),
\end{align*}
as $\epsilon \rightarrow 0$. Moreover, we have 
\begin{align*}
\frac{(-i)\pi RJ_n(\omega R)}{2} \delta \omega \left(\big(H_n^{(1)}\big)'(\omega R) - \frac{J_n'(\omega R)\big(H_n^{(1)}\big)'(\omega R_d)}{J_n'(\omega R_d)} \right) &= \begin{cases} \delta \left(1-\frac{R^2}{R_d^2} \right) + O(\delta\omega^2\ln\omega) \quad &n = 0 \\ O(\delta) &n\neq 0\end{cases} \\ 
&= \begin{cases} 2\delta\frac{\epsilon}{R} + O(\delta\epsilon + \delta\omega^2\ln\omega) \quad &n = 0 \\ O(\delta) &n\neq 0\end{cases}
\end{align*}

We are now ready to compute $\left\langle\Phi^*,\big(\A_{D}^\epsilon(\omega^*,\delta)-\A_{D}(\omega^*,\delta)\big){ \Psi^*} \right\rangle$. Using the expressions for $\Phi^*$ and $\Psi^*$, and using $(\omega^*)^2 = O(\delta)$, we have that as $\delta\rightarrow 0$ and $\epsilon\rightarrow 0$,
\begin{align*}
\left\langle\Phi^*,\big(\A_{D}^\epsilon(\omega^*,\delta)-\A_{D}(\omega^*,\delta)\big){ \Psi^*} \right\rangle &= \big\langle -\delta\psi_{\alpha^*}, -R\epsilon\ln\omega^* \psi_{\alpha^*} \big\rangle + \big\langle \chi_{\p D},2 a_0 \delta\frac{\epsilon}{R} \chi_{\p D}\big\rangle + O(\delta + \epsilon^2\delta\ln\delta),
\\ 
&= \delta\epsilon\ln\omega^*\left(R\|\psi_{\alpha^*}\|^2_{L^2(\p D)}  - 2\capacity_{D,\alpha^*} \right) + O(\epsilon\delta + \epsilon^2\delta\ln\delta).
\end{align*}
It remains to study the sign of the above expression. Define the factor $S = S(R)$ as
$$ S(R) = \left(R\|\psi_{\alpha^*}\|^2_{L^2(\p D)}  - 2\capacity_{D,\alpha^*} \right).$$
We will show that $S(R) < 0$ in the dilute case while $S(R) > 0$ in the non-dilute case.

\bigskip
\noindent\textit{Dilute case.} When the bubbles are small, we can compute $S(R)$ explicitly. Assume $R = \eta$ for some small $\eta$. Then we have \cite{bandgap}
$$\S_{D}^{\alpha,0} = \S_D + O(\eta).$$
Using the fact that $D$ is a circle, it is easily verified that
$$\|\psi_{\alpha^*}\|^2_{L^2(\p D)} = \frac{2\pi}{R(\ln R)^2} + O(\eta), \quad \capacity_{D,\alpha^*} = -\frac{2\pi}{\ln R} + O(\eta).$$
We then have 
$$\left(R\|\psi_{\alpha^*}\|^2_{L^2(\p D)}  - 2\capacity_{D,\alpha^*} \right) = \frac{2\pi}{\ln R}\left(\frac{1}{\ln R}+2\right),$$
which is negative for small $R$.

\bigskip
\noindent\textit{Non-dilute case.} We will show that $S(R)$ is positive for $R$ large enough. We decompose $\psi_{\alpha^*}$ as follows:
$$\psi_{\alpha^*} = c\chi_{\p D} + \phi_{\alpha_*}, \quad  \int_{\p D} \phi_{\alpha_*}\dx \sigma = 0.$$
Then it is easily verified that $$c = -\frac{\capacity_{D,\alpha^*}}{2\pi R},$$ so that $$\|\psi_{\alpha^*}\|^2_{L^2(\p D)} = \frac{\capacity_{D,\alpha^*}^2}{2\pi R} + \|\phi_{\alpha^*}\|^2_{L^2(\p D)} \geq \frac{\capacity_{D,\alpha^*}^2}{2\pi R}.$$
Consequently,
$$S(R) \geq \frac{\capacity_{D,\alpha^*}}{2\pi}\left(\capacity_{D,\alpha^*} - 4\pi \right).$$
Moreover, we have that $\capacity_{D,\alpha^*} \rightarrow \infty$ as $R\rightarrow 1/2$. Indeed, using the characterisation found in \cite{highfrequency}, we have that
$$\capacity_{D,\alpha^*} = \int_{Y\setminus D} |\nabla v_0|^2 \dx x,$$
where $v_0$ is the harmonic function in $Y\setminus D$ with boundary values
$$v_0 = \begin{cases} 1, \quad &\text{on } \p D \\ 0, & \text{on } \p Y.\end{cases}$$
It is easily verified that this integral is unbounded as $R\rightarrow 1/2$. Hence $S(R)$ is positive for $R$ close to $1/2$.
\qed

\section{Numerical illustration} \label{sec-4}

Here we provide numerical examples showing the existence of the subwavelength localized modes. 

\smallskip
\noindent\textbf{Asymptotic formula. }Observe that the $\delta$-error terms in Lemma \ref{lem:estim1} only differ from the leading-order terms by a factor of $\ln\delta$, making these expressions unsuitable for numerical computations. Using the same arguments, it is straight-forward to derive the following more precise asymptotic expressions:
$$
\big\langle \Phi^*, \frac{d}{d \omega}\A^*(\omega^*,\delta) \Psi^* \big\rangle = -2\pi{\omega^*R^3 \ln (\omega^*R)} - \pi\omega^*R^3 -\frac{\delta R}{\omega^*}\capacity_{D,\alpha^*} + O(\delta\ln\delta),
$$
and 
$$
\left\langle\Phi^*,\big(\A_{D}^\epsilon(\omega^*,\delta)-\A_{D}(\omega^*,\delta)\big){ \Psi^*} \right\rangle = \delta\epsilon \left(\ln R + 2\pi\eta_{\omega^*}\right) \left(R\|\psi_{\alpha^*}\|^2_{L^2(\p D)}  - 2\capacity_{D,\alpha^*} \right) + O(\epsilon\delta^{3/2}\ln\delta + \epsilon^2\delta\ln\delta).
$$

\smallskip
\noindent\textbf{Implementation. }To verify the asymptotic formula, we also compute the localised frequency by discretizing the operator $\mathcal{M}^\epsilon$ appearing in \eqref{eq:fg}. Recall that it is given by
$$
\mathcal{M}^\epsilon(\omega)=\bigg(I+\big(\A_{D}^\epsilon(\omega,\delta) - \A_{D}(\omega,\delta)\big)\frac{1}{(2\pi)^2}\int_{BZ}\A^\alpha(\omega,\delta)^{-1} d\alpha\bigg).
$$
We use the Fourier basis for the discretisation as in \cite{bandgap}. Specifically, we use $e^{i n\theta}, n=-N,\ldots,N$ as basis where $N$ is the truncation order. If we increase $N$, then we should get more accurate results.

The operators $\A_D$ and $\A_D^\epsilon$ are already represented in the Fourier basis in 
\eqref{eq:AD_multipole} and \eqref{eq:ADd}, respectively. We next consider the operator $\A^\alpha$. In \cite{bandgap}, it was shown that the operators $\mathcal{S}^{\alpha,k}$ and $\p \mathcal{S}_D^{\alpha,k}/\partial \nu \big|^\pm_{\p D}$ have the following matrix representations in the Fourier basis: 
$$
\ds\mathcal{S}_D^{\alpha,k}|_{\p D} \approx (S^\alpha_{m,n})_{m,n=-N}^N, 
\quad 
\frac{\p\mathcal{S}_D^{\alpha,k}}{\p \nu} \big|^\pm_{\p D} \approx
(\widetilde{S}^{\alpha,\pm}_{m,n})_{m,n=-N}^N, 
$$
where the matrix elements $S_{m,n}^\alpha$ and $\widetilde{S}^{\alpha}_{m,n}$ are given by
\begin{equation} \label{matrixa} \begin{array}{lll}
S^\alpha_{m,n} &=& \ds  -\frac{i\pi R}{2} J_n (k R)H_n^{(1)}(k R) \delta_{mn} - \frac{i\pi R}{2} J_n(k R)(-1)^{n-m} Q_{n-m}J_m(k R),
\\
\nm 
\widetilde{S}^{\alpha,\pm}_{m,n} &=&\ds   \pm \frac{1}{2}- k \frac{i\pi R}{2}\Big(J_n \cdot (H_n^{(1)})'+J_n'\cdot H_n^{(1)}\Big)(k R) \delta_{mn}- \frac{i\pi R}{2}J_n(k R)(-1)^{n-m} Q_{n-m}k J_m'(k R).
\end{array} \end{equation}
Here, $Q_n$ is so called the lattice sum defined by
$$
Q_n := \sum_{m\in \mathbb{Z}^2, m\neq 0} H_n^{(1)}(k|m|)e^{i n \arg(m)}e^{i m\cdot \alpha}.
$$
%So, from \eqnref{SingleLayer_multipole} and \eqnref{QuasiSingleLayer_multipole}, we finally get the explicit representation of $\mathcal{S}_D^{\alpha,k}$. 
For an efficient method for computing the lattice sum $Q_n$, we refer the reader to  \cite{Linton2010}.
From (\ref{matrixa}),  the matrix representation of $\mathcal{A}^\alpha$ immediately follows. We used $N=7$ for the truncation order of the Fourier basis.

It remains to consider the integral $\int_{BZ} (\mathcal{A}^\alpha)^{-1} d\alpha$. We compute the two-dimensional integration with respect to $\alpha$ using Gauss quadrature. Then, we finally get the discretization $(M^\epsilon_{m,n})$ of the operator $\mathcal{M}^\epsilon(\omega)$. 
To find the resonant frequency for the localized mode, we apply Muller's method \cite{lecturenotes} to the determinant of the matrix $(M_{m,n}^\epsilon)$. The maximum frequency $\omega^*$ of the first band can be used as a  good initial guess. The band structure of the unperturbed crystal is computed exactly as in \cite{bandgap}. 

For the material parameters, we use $\rho_w=\kappa_w=5000$ and $\rho_b=\kappa_b=1$. In this case, we have $\delta=0.0002$. Moreover, in both examples below we consider the two cases $\epsilon = -0.03R$ and $\epsilon = 0.03R$. 

\begin{figure*} %[!h]
	\begin{center}
		\includegraphics[height=5cm]{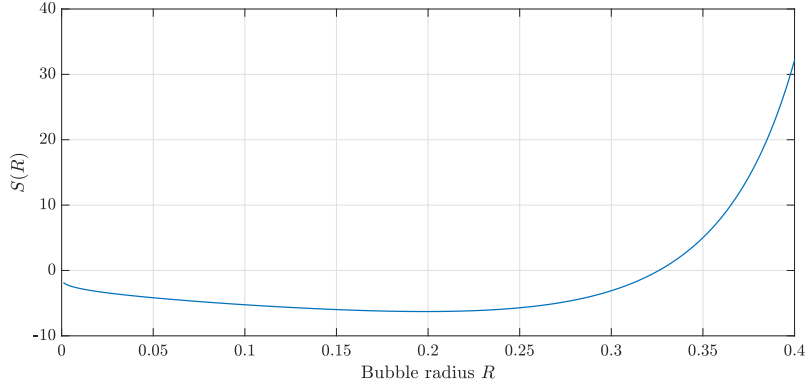}
		\caption{Illustration of $S(R)$ as a function of $R$. The zero $S(R_0) = 0$ is attained at $R_0 = 0.326$.}
		\label{fig:SR}
	\end{center}
\end{figure*}

\begin{figure*} %[!h]
\begin{center}
 \includegraphics[height=5.0cm]{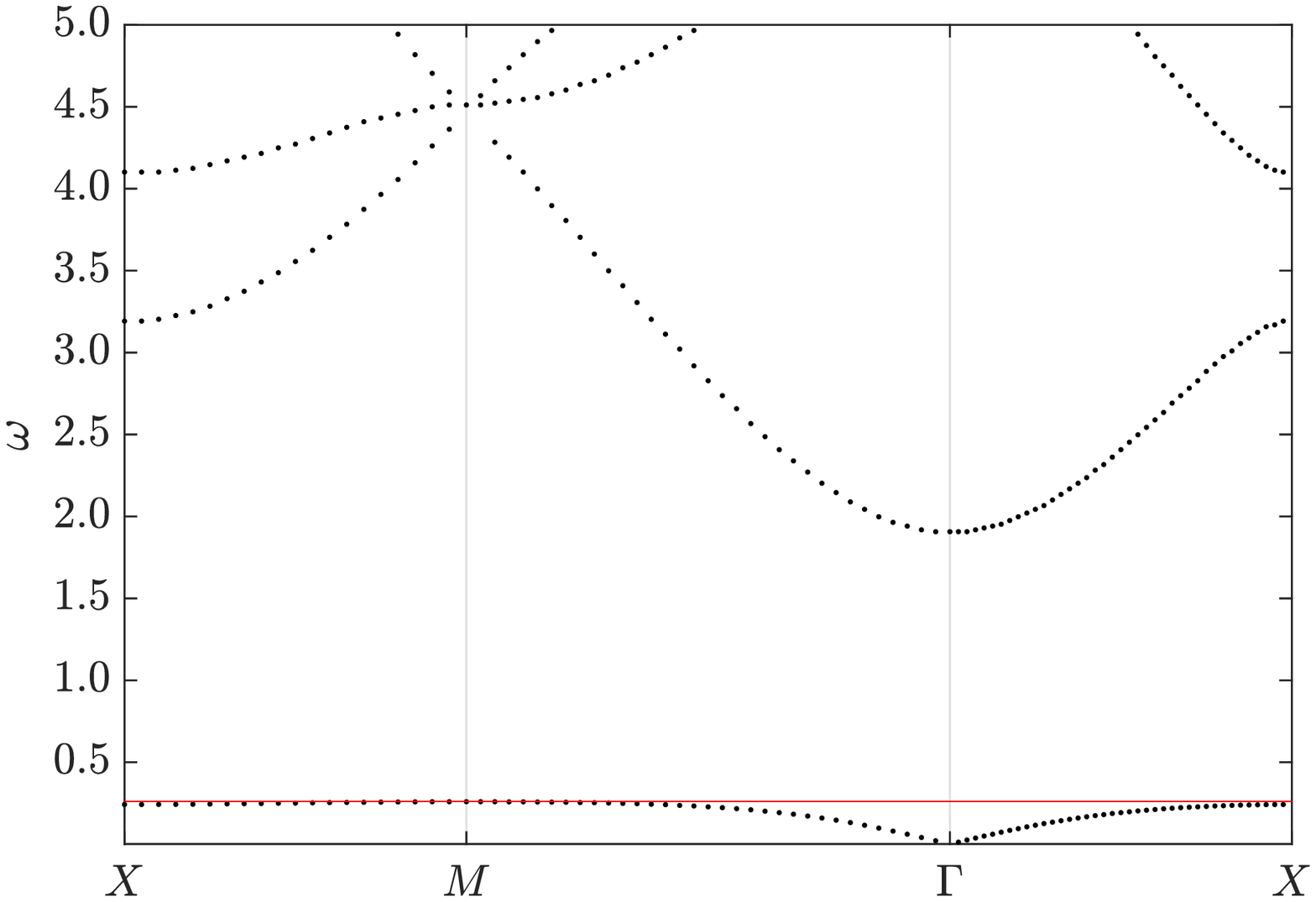}
 \hspace{0.1cm}
 \includegraphics[height=5.0cm]{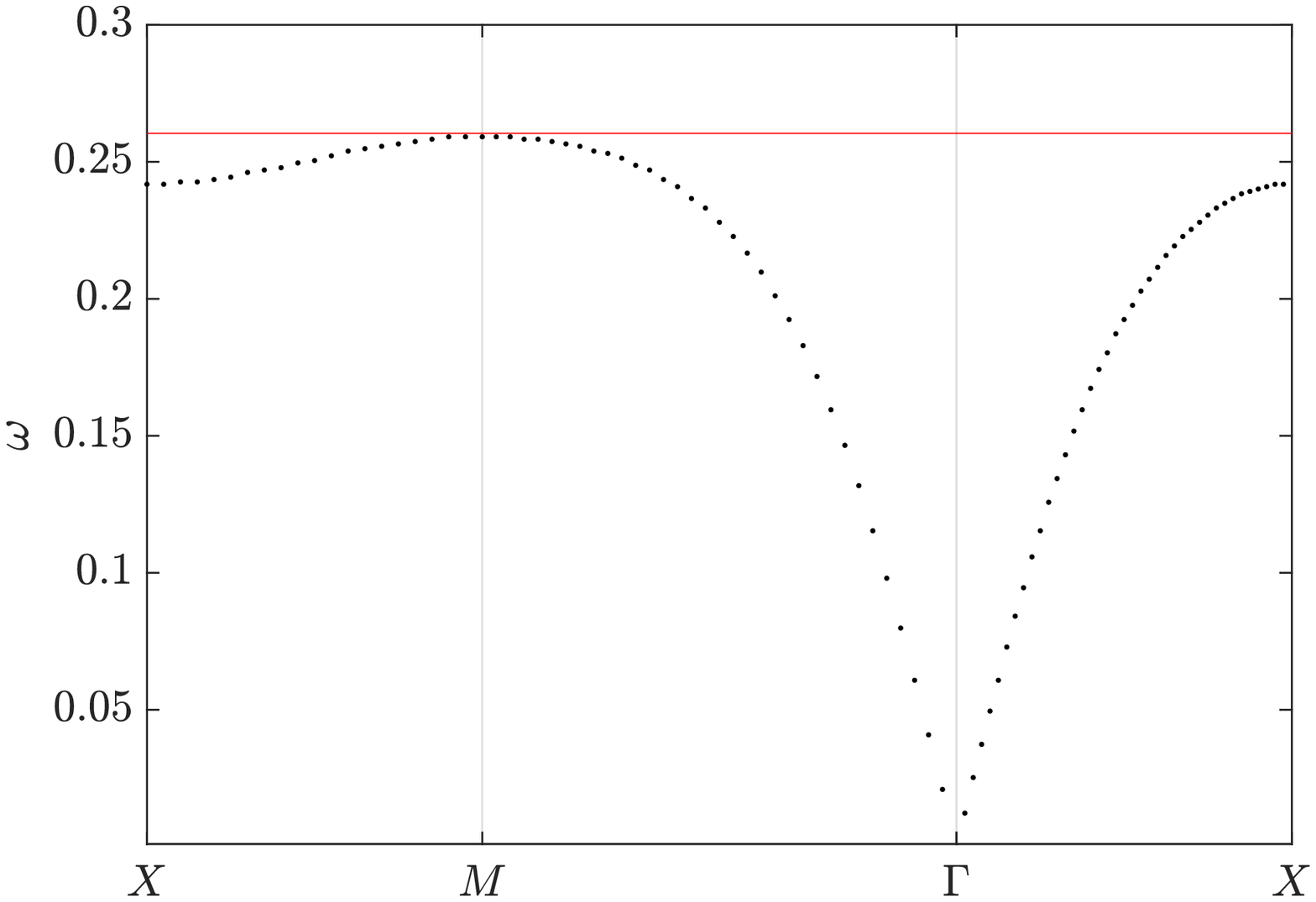}
 \caption{ (Dilute case)
	Band structure of the unperturbed crystal (black circles) and frequency of localized mode for the perturbed crystal $\mathcal{C}_d$ (red line), with bubble radius $R=0.05$ and $\epsilon = -0.03 \times R = -0.0015$.}
  \label{fig:bandgap_dilute}
\end{center}
\end{figure*}

In Figure \ref{fig:SR}, we show the function $S(R)$ over the range $R\in (0,0.4)$. As it can be seen, the function has a zero at $R_0 = 0.326\approx 1/3$. For $R$ below this zero, the crystal is dilute and we expect that $\epsilon$ must be negative in order to have a defect frequency in the bandgap. For $R$ above this zero, we expect that $\epsilon$ must be positive.

%%%
\smallskip
\noindent\textbf{Example 1. }We first consider the dilute case with $R=0.05$. In Figure \ref{fig:bandgap_dilute}, we show the computed band structure for this case and the frequency $\omega^\epsilon$ for the localized mode in the case $\epsilon = -0.03R$. In the case $\epsilon =0.03R$, no singular values of $\mathcal{M}^\epsilon(\omega)$ were found. The black circles represents the band structure of the unperturbed crystal $\C$. The horizontal red line corresponds to the frequency $\omega^\epsilon$ for the localized mode. The points $\Gamma, X$ and $M$ represent $\alpha=(0,0)$, $\alpha=(\pi,0)$, and $\alpha=(\pi,\pi)$, respectively.

In Figure \ref{fig:bandgap_dilute} (right), we plot the magnified subwavelength band. In the first band, the maximum of $\omega$ is attained at the point $M$ (or $\alpha=\alpha^*=(\pi,\pi)$).
The maximum frequency $\omega^*$ of the first band is $\omega^*\approx0.2591$. The defect frequency computed by discretizing the operator is $\omega^\epsilon_d\approx 0.2592$, while the defect frequency computed using the asymptotic formula is $\omega^\epsilon_a\approx 0.2604$. Observe that the defect frequency is exponentially close to $\omega^*$, leading to a high relative error in the computations. However, the frequency $\omega^\epsilon$ is clearly located in the bandgap and in the subwavelength regime. Therefore, the numerical results indicate the existence of a subwavelength localized mode when $\epsilon <0$ but not when $\epsilon >0$.

\smallskip
\noindent\textbf{Example 2. } We next consider the non-dilute case with $R=0.45$. In Figure \ref{fig:bandgap_nondilute}, we show the computed band structure for this non-dilute case, and the frequency $\omega^\epsilon$ for the localized mode in the case $\epsilon = 0.03R$. No defect frequency was found for $\epsilon = -0.03R$ The maximum frequency $\omega^*$ of the first band is $\omega^*\approx0.10806$. The defect frequency computed by discretizing the operator is $\omega^\epsilon_d\approx 0.10807$, while the defect frequency computed using the asymptotic formula is $\omega^\epsilon_a\approx 0.10844$. Again, the frequency $\omega^\epsilon$ is located above the subwavelength band.

\section{Concluding remarks} \label{sec-5}
In this paper, we have proved for the first time the possibility of localizing waves at the deep subwavelength scale. We have considered a bubbly crystal and produced a localized mode by perturbing the size of one bubble. Moreover, we have proven that the sign of the perturbation depends of the crystal: in the dilute regime the defect bubble should be smaller than the unperturbed bubbles, while in the non-dilute regime the defect bubble should be larger. We have illustrated our findings with  numerical experiments. Our results in this paper shed light on the mechanism behind the control and guiding of waves at deep subwavelength scales. In forthcoming works, we plan to investigate the robustness of such subwavelength defect modes with respect to spatial disorder.
We will also consider the possibility of guiding waves at deep subwavelength scales using line defects. Another challenging problem we plan to tackle is to prove or disprove the existence of fundamental limits on the Purcell factor in subwavelength bandgap materials.  

\begin{figure*}
\begin{center}
 \includegraphics[height=5.0cm]{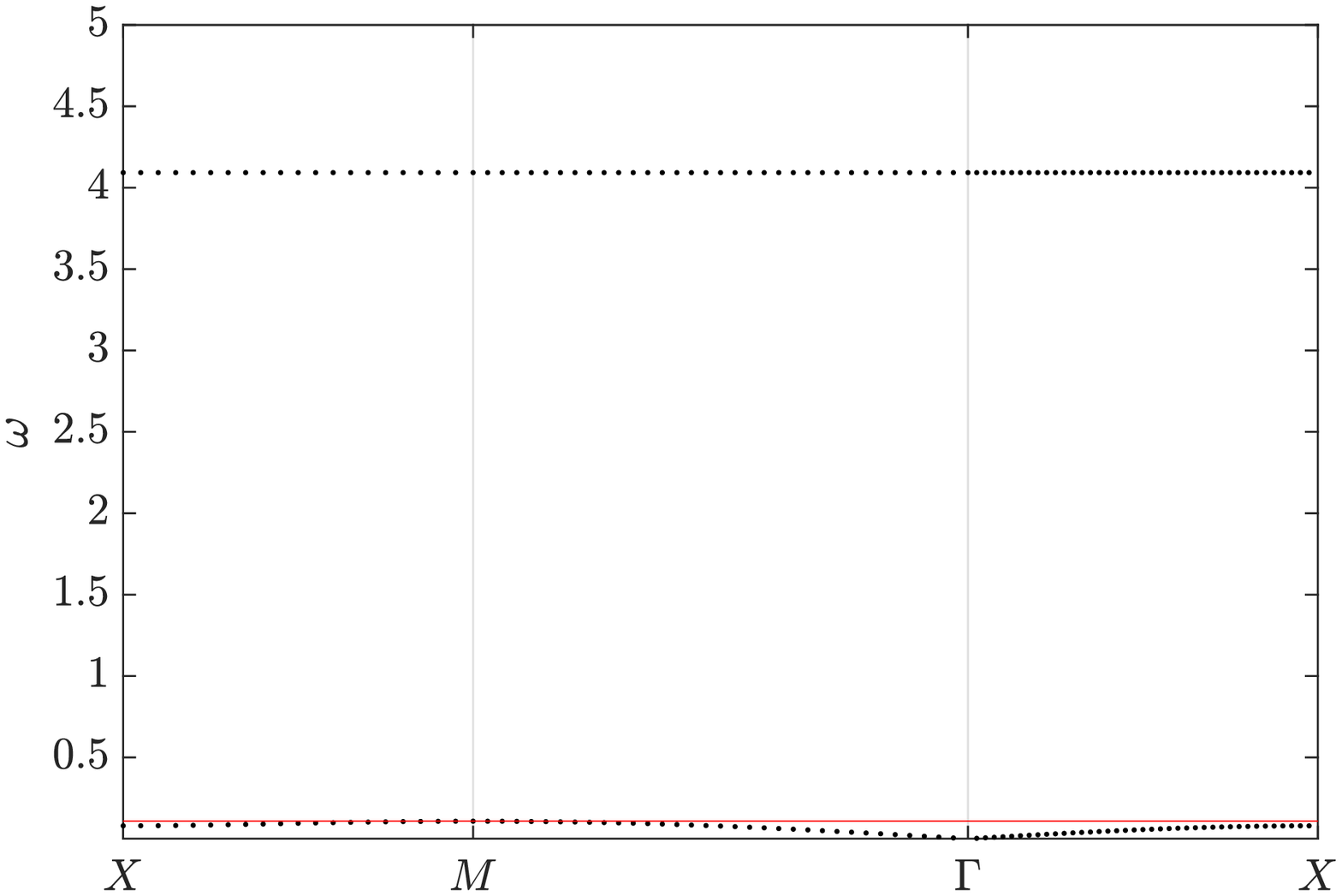}
 \hspace{0.1cm}
 \includegraphics[height=5.0cm]{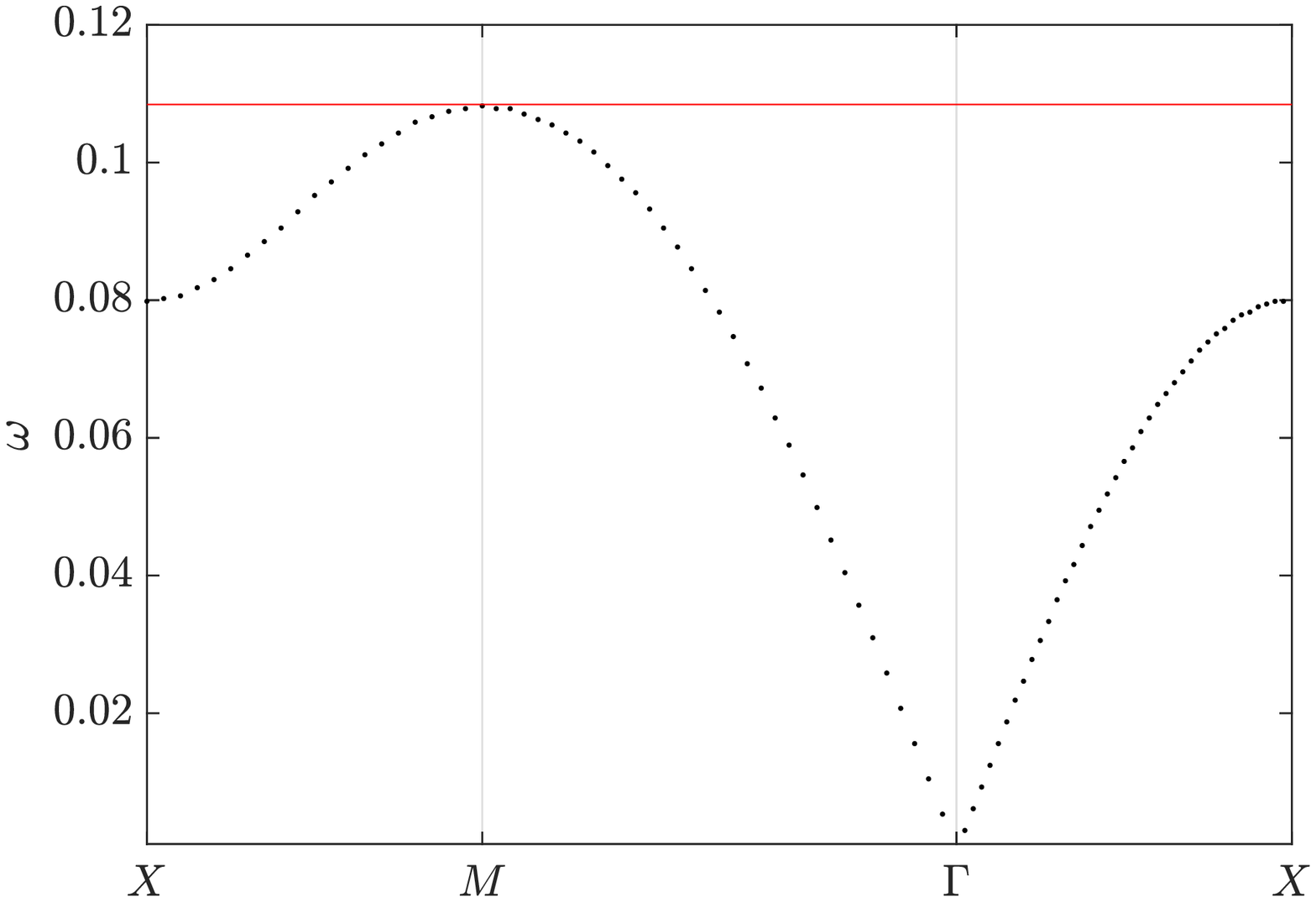}
   \caption{ (Non-dilute case)
  Band structure of the unperturbed crystal (black circles) and frequency of localized mode for the perturbed crystal $\mathcal{C}_d$ (red line), with bubble radius $R=0.45$ and $\epsilon = 0.03 \times R = 0.0135$.}
  \label{fig:bandgap_nondilute}
\end{center}
\end{figure*}

\bibliography{defect_final}{}

\begin{thebibliography}{10}

\bibitem{first}
H.~{Ammari}, B.~{Fitzpatrick}, D.~{Gontier}, H.~{Lee}, and H.~{Zhang}.
\newblock {Minnaert resonances for acoustic waves in bubbly media}.
\newblock {\em Annales de l'Institut Henri Poincaré: Analyse Nonlinéaire, to
  appear}.

\bibitem{superfocusing}
H.~{Ammari}, B.~{Fitzpatrick}, D.~{Gontier}, H.~{Lee}, and H.~{Zhang}.
\newblock {Sub-wavelength focusing of acoustic waves in bubbly media}.
\newblock {\em Proc. Royal Soc. A}, 473:20170469, 2017.

\bibitem{doublenegative}
H.~{Ammari}, B.~{Fitzpatrick}, H.~{Lee}, S.~{Yu}, and H.~{Zhang}.
\newblock {Double-negative acoustic metamaterials}.
\newblock {\em ArXiv e-prints}.

\bibitem{bandgap}
H.~{Ammari}, B.~{Fitzpatrick}, H.~{Lee}, S.~{Yu}, and H.~{Zhang}.
\newblock {Subwavelength phononic bandgap opening in bubbly media}.
\newblock {\em J. Differential Equations}, 263(9):5610--5629, 2017.

\bibitem{Ammari2007_polarizationBook}
H.~Ammari and H.~Kang.
\newblock {\em Polarization and moment tensors: with applications to inverse
  problems and effective medium theory}, volume 162.
\newblock Springer Science \& Business Media, 2007.

\bibitem{arma}
H.~Ammari, H.~Kang, and H.~Lee.
\newblock Asymptotic analysis of high-contrast phononic crystals and a
  criterion for the band-gap opening.
\newblock {\em Arch. Ration. Mech. Anal.}, 193(3):679--714, 2009.

\bibitem{lecturenotes}
H.~Ammari, H.~Kang, and H.~Lee.
\newblock {\em Layer Potential Techniques in Spectral Analysis}, volume 153 of
  {\em Mathematical Surveys and Monographs}.
\newblock American Mathematical Society, Providence, 2009.

\bibitem{soussi}
H.~Ammari, H.~Kang, S.~Soussi, and H.~Zribi.
\newblock Layer potential techniques in spectral analysis. ii. sensitivity
  analysis of spectral properties of high contrast band-gap materials.
\newblock {\em Multiscale Model. Simul.}, 5(2):646--663, 2006.

\bibitem{highfrequency}
H.~{Ammari}, H.~{Lee}, and H.~{Zhang}.
\newblock {High frequency homogenization of bubbly crystals}.
\newblock {\em arXiv, 1708.07955}.

\bibitem{santosa}
H.~Ammari and F.~Santosa.
\newblock {Guided waves in a photonic bandgap structure with a line defect}.
\newblock {\em SIAM J. Appl. Math.}, 64(6):2018--2033, 2004.

\bibitem{Ammari2015}
H.~Ammari and H.~Zhang.
\newblock A mathematical theory of super-resolution by using a system of
  sub-wavelength {H}elmholtz resonators.
\newblock {\em Commun. in Math. Phys.}, 337(1):379--428, 2015.

\bibitem{Ammari2015_a}
H.~Ammari and H.~Zhang.
\newblock Super-resolution in high-contrast media.
\newblock {\em Proc. R. Soc. A}, 471(2178), 2015.

\bibitem{thinlayer}
Habib {Ammari}, Brian {Fitzpatrick}, Hyundae {Lee}, Erik {Orvehed Hiltunen},
  and Sanghyeon {Yu}.
\newblock {Subwavelength resonances of encapsulated bubbles}.
\newblock {\em arXiv e-prints}, page arXiv:1810.12756, Oct 2018.

\bibitem{weinstein3}
V.~Duchêne, I.~Vukićević, and M.~I. Weinstein.
\newblock {Homogenized description of defect modes in periodic structures with
  localized defects}.
\newblock {\em Commun. Math. Sci.}, 13(3):777--823, 2015.

\bibitem{weinstein2}
V.~Duchêne, I.~Vukićević, and M.I. Weinstein.
\newblock {Oscillatory and localized perturbations of periodic structures and
  the bifurcation of defect modes}.
\newblock {\em SIAM J. Math. Anal.}, 47(5):3832--3883, 2015.

\bibitem{figotin1}
A.~Figotin and V.~Goren.
\newblock { Resolvent method for computations of localized defect modes of
  H-polarization in two-dimensional photonic crystals}.
\newblock {\em Phys. Rev. E}, 64(3):056623, 16 pp, 2001.

\bibitem{figotin}
A.~Figotin and A.~Klein.
\newblock Localized classical waves created by defects.
\newblock {\em Journal of Statistical Physics}, 86(1):165--177, Jan 1997.

\bibitem{figotin2}
A.~Figotin and A.~Klein.
\newblock {Midgap defect modes in dielectric and acoustic media}.
\newblock {\em SIAM J. Appl. Math.}, 58(6):1748--1773, 1998.

\bibitem{figotinkuchment}
A.~Figotin and P.~Kuchment.
\newblock {Spectral properties of classical waves in high-contrast periodic
  media}.
\newblock {\em SIAM J. Appl. Math.}, 58(2):683--702, 1998.

\bibitem{Gohberg1971}
I.C. Gohberg and E.I. Sigal.
\newblock An operator generalization of the logarithmic residue theorem and the
  theorem of {R}ouch\'{e}.
\newblock {\em Sb. Math.}, 13(4):603--625, 1971.

\bibitem{hempel}
R.~Hempel and K.~Lienau.
\newblock Spectral properties of periodic media in the large coupling limit.
\newblock {\em Comm. Partial Differential Equations}, 25:1445--1470, 2000.

\bibitem{weinstein1}
M.A. Hoefer and M.I. Weinstein.
\newblock {Defect modes and homogenization of periodic Schrödinger operators}.
\newblock {\em SIAM J. Math. Anal.}, 43(2):971--996, 2011.

\bibitem{klein}
A.~Klein and A.~Koines.
\newblock { A general framework for localization of classical waves. I.
  Inhomogeneous media and defect eigenmodes}.
\newblock {\em Math. Phys. Anal. Geom.}, 4(2):97--130, 2001.

\bibitem{kuchment}
P.~Kuchment.
\newblock {\em {Floquet Theory for Partial Differential Equations}}.
\newblock Number~60 in Operator Theory: Advances and Applications.
  {Birkh\"auser Verlag}, Basel, 1993.

\bibitem{kuchment2}
P.~Kuchment.
\newblock {An overview of periodic elliptic operators}.
\newblock {\em Bull. Amer. Math. Soc.}, 53(3):343--414, 2016.

\bibitem{experiment2013}
F.~Lemoult, N.~Kaina, M.~Fink, and G.~Lerosey.
\newblock Wave propagation control at the deep subwavelength scale in
  metamaterials.
\newblock {\em Nature Physics}, 9:55--60, 2013.

\bibitem{experiment}
F.~Lemoult, N.~Kaina, M.~Fink, and G.~Lerosey.
\newblock Soda cans metamaterial: A subwavelength-scaled phononic crystal.
\newblock {\em Crystals}, 6(7), 2016.

\bibitem{lin1}
J.~Lin.
\newblock {A perturbation approach for near bound-state resonances of photonic
  crystal with defect}.
\newblock {\em European J. Appl. Math.}, 27(1):66--86, 2016.

\bibitem{lin2}
J.~Lin and F.~Santosa.
\newblock {Resonances of a finite one-dimensional photonic crystal with a
  defect}.
\newblock {\em SIAM J. Appl. Math.}, 73(2):1002--1019, 2013.

\bibitem{santosa2}
S.-Y. Lin, E.~Chow, V.~Hietala, P.R. Villeneuve, and J.D. Joannopoulos.
\newblock {Experimental demonstration of guiding and bending of electromagnetic
  waves in a photonic crystal}.
\newblock {\em Science}, 282:274--276, 1998.

\bibitem{Linton2010}
C.~M. Linton.
\newblock Lattice sums for the helmholtz equation.
\newblock {\em SIAM Rev.}, 52(4):630--674, 2010.

\bibitem{Lipton}
R.~Lipton and R.~Jr. Viator.
\newblock { Creating band gaps in periodic media}.
\newblock {\em Multiscale Model. Simul.}, 15(4):1612--1650, 2017.

\bibitem{Lipton2}
R.~Lipton and R.~Jr. Viator.
\newblock {Bloch waves in crystals and periodic high contrast media}.
\newblock {\em ESAIM Math. Model. Numer. Anal.}, 51(3):889--918, 2017.

\bibitem{phononic1}
Z.~Liu, , X.~Zhang, Y.~Mao, Y.Y. Zhu, Z.~Yang, C.T. Chan, and P.~Sheng.
\newblock Locally resonant sonic materials.
\newblock {\em Science}, 289(5485):1734--1736, 2000.

\bibitem{pnas}
K.H. Matlack, A.~Bauhofer, S.~Kr\"odel, A.~Palermo, and C.~Daraio.
\newblock Composite 3d-printed metastructures for low-frequency and broadband
  vibration absorption.
\newblock {\em Proc. Natl. Acad. Sci. USA}, 113(30):8386--8390, 2016.

\bibitem{Minnaert}
M.~Minnaert.
\newblock {{O}n musical air-bubbles and the sounds of running water}.
\newblock {\em The London, Edinburgh, Dublin Philos. Mag. and J. of Sci.},
  16:235--248, 1933.

\bibitem{purcell}
E.M. Purcell.
\newblock {Spontaneous emission probabilities at radio frequencies}.
\newblock {\em Phys. Rev.}, 69:674, 1946.

\bibitem{MMMP4}
M.~Reed and B.~Simon.
\newblock {\em Methods of Modern Mathematical Physics IV: Analysis of
  Operators}.
\newblock Academic Press Inc., California, USA, 1978.

\bibitem{nature}
E.L. Thomas.
\newblock Bubbly but quiet.
\newblock {\em Nature}, 462(24), 2009.

\bibitem{Wilcox}
S.~Wilcox, L.~C. Botten, R.~C. McPhedran, C.~G. Poulton, and C.~Martijn
  de~Sterke.
\newblock Modeling of defect modes in photonic crystals using the fictitious
  source superposition method.
\newblock {\em Phys. Rev. E}, 71:056606, May 2005.

\end{thebibliography}
\bibliographystyle{plain}
\end{document}